\documentclass[11pt]{article}

\usepackage{my_preamble}
\usepackage{fullpage}
\usepackage{enumerate}
\usepackage{stmaryrd}

\theoremstyle{remark}

\renewcommand{\qtp}{\mathrm{qftp}}

\newcommand{\Age}{\textsf{Age}}
\newcommand{\XX}{\mathbb{X}}
\newcommand{\GG}{\mathbf{G}}

\newcommand{\sig}{\mathrm{sig}}

\renewcommand{\TT}{\mathbb{T}}

\renewcommand{\KK}{\mathbf{K}}
\renewcommand{\K}{\mathscr{K}}

\newcommand{\G}{\mathscr{G}}

\newcommand{\FF}{\mathbf{F}}

\newcommand{\Fin}{\mathbf{Fin}}
\newcommand{\cleq}{\trianglelefteq}

\newcommand{\emb}{\mathrm{Emb}}

\newcommand{\pphi}{{\boldsymbol\phi}}

\newcommand{\img}{\mathrm{img}}
\newcommand{\dom}{\mathrm{dom}}

\newcommand{\sym}{\mathrm{Sym}}
\newcommand{\HH}{\mathbf{H}}
\renewcommand{\H}{\mathcal{H}}
\newcommand{\JJ}{\mathbf{J}}
\newcommand{\qf}{\mathsf{qf}}
\newcommand{\cclass}{\mathscr{C}}
\renewcommand{\SS}{\mathbf{S}}
\newcommand{\fs}{\textsc{fs}}
\newcommand{\MO}{\mathbf{MO}}

\newcommand{\Th}{\mathrm{Th}}

\begin{document}

\title{On positive local combinatorial dividing-lines in model theory}
\author{Vincent Guingona and Cameron Donnay Hill}
\date{\today}
\maketitle

\begin{abstract}
We introduce the notion of positive local combinatorial dividing-lines in model theory.  We show these are equivalently characterized by indecomposable algebraically trivial \fraisse classes and by complete prime filter classes.  We exhibit the relationship between this and collapse-of-indiscernibles dividing-lines.  We examine several test cases, including those arising from various classes of hypergraphs.
\end{abstract}

\section{Introduction}\label{Sec_Intro}

There are several productive approaches to identifying dividing-lines in model theory -- separating ``wild'' from ``tame'' first-order theories -- and in this paper, we attempt to formalize one of these approaches. These approaches to dividing-lines come in roughly three flavors: one based on geometric properties (formulated by way of model-theoretic independence relations), the Keisler order or saturation-of-ultrapowers, and positive local combinatorial properties. The last of these is, of course, among the oldest in the modern setting, including the stability/instability and N/IP dividing-lines, but it has not previously been placed in a general context. In this paper, we develop just such a general context/formulation for positive local combinatorial dividing-lines. 

In this introduction, we will first try to provide some context for our work, hopefully demonstrating how the viewpoint developed here differs meaningfully from the others and captures quite different tameness/wildness separations. 
To start with, one might focus on dividing-lines based on properties of model-theoretic independence relations (like non-forking or non-\th-forking independence), which we will call {\em geometric dividing-lines}. Geometric dividing-lines include in/stability, non/simplicity, and non/rosiness. In practice, model-theoretic independence relations are often formulated in a formula-by-formula manner, but especially when viewed in an abstract setting as in \cite{adler-2007}, model-theoretic independence relations really deal with partial and complete types in a ``holistic'' way. In some sense, model-theoretic independence relations have at least aesthetic connections to the idea of a generic point of an algebraic variety, so it is not unreasonable to view geometric dividing-lines as capturing {\em global} properties of theories. These considerations qualitatively separate geometric dividing-lines from the other kinds that we will discuss, so we take the view that geometric dividing-lines are outside the scope of this paper. By contrast, both the saturation-of-ultrapowers program (as shown in \cite{malliaris-2009}) and our program on positive local combinatorial dividing-lines are explicitly local in nature -- both examine and compare theories on a formula-by-formula basis. (Furthermore, the discussion of geometric dividing-lines seems to require a strong {\em a priori} assumption that the theories in question actually have model-theoretic independence relations, at least in a weak sense, to begin with. Thus, if one wished to {\em compare} two arbitrary theories geometrically, a fairly deep program on how to do this at all would probably be required.)

In order to compare saturation-of-ultrapowers dividing-lines and those based on positive local combinatorics, it will be helpful begin with an example.
\begin{itemize}
\item For each $0<k<\omega$, let $P_k = \big\{\left\{0,1,...,\ell-1\right\}:\ell\leq k\big\}$. (The $P$ here is for ``positive.'') Then, to say that a theory $T$ has the order property is to say that there is a formula $\phi(x,y)$ of the language of $T$ that ``represents'' or ``realizes'' every one of the configurations $(k,P_k)$:
\begin{quote}
{\em For every $k$, there are $a_0,...,a_{k-1}$ and $b_0,...,b_{k-1}$ such that $\models \phi(a_i,b_\ell)\,\iff\,i\in\{0,1,...,\ell-1\}$.}
\end{quote}
In general, after settling on a formula $\phi(x_0,...,x_{n-1})$ (or several of them), realization of a P-LC is governed by nothing more than $\pm\phi$ and selection of a set of parameters from a model of the theory. 

\item Now, let $N_k = \P(k)\setminus P_k$, where the $N$ is for ``negative.'' To say that $T$ has the strict order property (SOP) is to say that there is a formula $\phi(x,y)$ of the language of $T$ that represents/realizes every one of the configurations $(k,P_k,N_k)$, meaning that $(k,P_k)$ is realized while simultaneously avoiding $N_k$ in a global sense:
\begin{quote}
{\em  For every $k$, there are $a_0,...,a_{k-1}$ and $b_0,...,b_{k-1}$ such that 
\begin{align*}
{\models}{\phi(a_i,b_\ell)}\,&\iff\,i\in\{0,1,...,\ell-1\},\\
\ell_1<\ell_2\leq k\,&\implies\,\models\neg\exists x\big(\phi(x,b_{\ell_1})\wedge \neg\phi(x,b_{\ell_2})\big).
\end{align*}
}
\end{quote}
That is, in addition to requiring that $(a_i)_{i<k}$, $(b_\ell)_{\ell\leq k}$ realize the P-LC $(k,P_k)$ via $\phi$, we require that {\em no re-ordering of $(b_\ell)_{\ell\leq k}$} can be used to realize $(k,P_k)$ via $\phi$.
\end{itemize}
 Thus, among the positive local combinatorial properties (P-LC properties) are the order property (OP) and the multi-order properties ($\mathrm{MOP}_n$) of \cite{guingona-hill-2015}, the independence property (IP) and the $r$-independence properties ($r$-IP) for all $0<r<\omega$. Among the positive/negative local combinatorial properties (P/N-LC properties) are the strict order property, the various tree properties (TP, TP1, and TP2), and the finite cover property (FCP).

Perhaps, it is suggestive that all of the P/N-LC properties listed above have, at least loosely, something to do with the saturation-of-ultrapowers program: FCP and SOP were the ``original'' beacons in the Keisler order, and much of the recent work in this area, \cite{malliaris-shelah-2013,malliaris-shelah-2016, malliaris-shelah-2016b} for example, focuses on (classes of) simple theories -- i.e. the NTP theories. Speaking very loosely, it appears that a non-saturation result requires something like a P/N-LC property in order to build an unrealized type into a regular ultrapower, although even canonically P-LC properties can carry implicit P/N-LC properties that allow such constructions.   Empirically speaking, the close connection between the saturation-of-ultrapowers program and P/N-LC properties suggests that in order to move away from that program, we do need to avoid P/N-LC properties, and this suggests that, barring geometric dividing-lines, P-LC properties are a good place to start.
A more practical issue in the saturation-of-ultrapowers program is that the Keisler order views even characteristic theories very differently compared to other points of view, such as the collapse-of-indiscernibles approach taken in \cite{guingona-hill-scow}, a strong form of the positive local combinatorial paradigm (though not entirely compatible with our work here). In that paper, it is shown that $T_\mathrm{feq}$, the model-companion of the theory of parametrized equivalence relations, is characteristic of non-rosy theories under the collapse-of-indiscernibles viewpoint, while in \cite{malliaris-2012}, $T_\mathrm{feq}$ arises as the $\leq_\textsc{k}$-minimum theory with TP2 (where $\leq_\textsc{k}$ denotes the Keisler order).
A second practical issue is that it does not ``see'' dividing-lines among unstable NIP theories. Thus, a different approach is needed if we want to include and differentiate between ordered theories.

However, as in the saturation-of-ultrapowers program, we will base our viewpoint on dividing-lines in this paper on an ordering $\cleq$ of theories, and we essentially identify dividing-lines with upward-closed classes of theories relative to $\cleq$. Speaking roughly, we take $T_1\cleq T_2$ to mean that, up to matching formulas of $T_1$ to formulas of $T_2$, ``every'' finite configuration that appears in a model of $T_1$ also appears in a model of $T_2$. (We write ``every'' because we will conscientiously avoid involving algebraic configurations.) Then to specify a dividing-line is to specify a class $\C$, upward-closed under $\cleq$, of ``wild'' theories; depending on context, $\C$ could be the class of unstable theories, the class theories with IP, and so forth. Subsequently, $N\C:=\TT\setminus \C$ (where $\TT$ is the class of all complete theories) comprises the ``tame'' theories with good-behavior properties to exploit. Our formulation of $\cleq$ is almost immediately linked to the idea of a family of \fraisse classes involved in a theory. 

In connection with collapse-of-indiscernibles, an immediate question is then, ``Which $\cleq$-dividing-lines $\C$ have a single characteristic \fraisse class?'' A second question is, ``What are the `irreducible' $\cleq$-dividing-lines, and what does it mean for $\C$ to be irreducible?'' Addressing the second part of the second question, we try (and hopefully succeed) to capture irreducibility in the notion of a complete prime filter class. The classes of unstable theories, of theories with IP, and of theories with $r$-IP ($0<r<\omega$) are all complete prime filter classes. We prove (Theorem \ref{Thm_prime_equals_fraisse}) that the complete prime filter classes are {\em precisely} the classes of theories that have a characteristic indecomposable \fraisse class. It follows that if $\mathscr{F}$ is the family of all complete prime filter classes -- i.e. the class of all irreducible $\cleq$-dividing-lines -- then $|\mathscr{F}|\leq 2^{\aleph_0}$.

We also carry out a detailed analysis of the \fraisse classes $\HH_r$ of $r$-hypergraphs (for each $2\leq r<\omega$) in the context of $\cleq$. These are all indecomposable \fraisse classes, and one result of this analysis is the finding that $\mathscr{F}$ is infinite; in fact, it contains an infinite strictly-nested chain (cf. \cite{malliaris-shelah-2016}). We also find that excluding/forbidding a $k$-clique ($k>r$, obtaining another \fraisse class $\HH_{r,k}$) has no effect on the associated dividing-line. Somewhat surprisingly, adding an order to $\HH_r$ in an unconstrained way, obtaining the class $\HH_r^<$ of ordered $r$-hypergraphs, also has no effect on the associated dividing-line. This is a special case of a general phenomenon: we prove that for a \fraisse class $\KK$, if the generic theory $T_\KK$ is unstable, then $\KK$ and $\KK^<$ characterize the same dividing-line. We also find that adding additional symmetric irreflexive relations does not change the associated dividing-line; that is to say, for a \fraisse class $\SS$ of societies (in the terminology of \cite{nesetril-rodl-1989}), $\SS$ characterizes the same dividing-line as does $\HH_r$, where $r$ is the maximum arity of a relation symbol in the language of $\SS$.

The work in this paper originated in an attempt to build a general framework for the collapse-of-indiscernibles dividing-lines addressed in \cite{guingona-hill-scow} and \cite{scow-2011}. That attempt failed insofar as indecomposable \fraisse class need not be Ramsey classes or even Ramsey-expandable (as far as we know currently). 
However, we do manage to prove that any dividing-line arising from an unstable Ramsey-expandable \fraisse class {\em is} a collapse-of-indiscernibles dividing-line. Since all of the $\HH_r$'s are unstable Ramsey-expandable \fraisse classes, it follows that there is an infinite strict chain of collapse-of-indiscernibles dividing-lines, all of which are irreducible.

%
%

%

\bigskip

\subsection{Outline of the paper}

\begin{description}
\item[Section \ref{Sec_Intro}:] In the remainder of this section, we introduce some notation, basic definitions, and conventions that are used throughout the rest of the paper.
\item[Section \ref{Sec_LC_ord}:] We introduce the infrastructure both for discussing the (positive) local combinatorics of a first-order theory and for comparing two theories based on their local combinatorics, arriving at the ordering $\cleq$. Around this ordering, we define precisely what is meant by a ``complete prime filter class'' after discussing at some length why, we believe, this formulation captures the idea of an irreducible dividing-line. We also state and prove Theorem \ref{Thm_prime_equals_fraisse}, showing that the irreducible dividing-lines are precisely the ones characterized by a single indecomposable \fraisse class. In all of our definitions to this point, we will have used only one-sorted languages for \fraisse classes, and to conclude this section, we show that there would, in fact, be nothing gained by allowing multiple sorts.

\item[Section \ref{Sec_LO_Indisc}:] In this section, we reconnect with the collapse-of-indiscernibles dividing-lines that impelled this project to begin with. It is a fact that the generic model of an algebraically trivial \fraisse class admits a generic linear ordering. It is also a fact that the generic model of a Ramsey class carries a 0-definable linear ordering, and many of the best-known Ramsey classes are obtained just by adding a linear order ``freely'' to a \fraisse class whose generic model is NSOP (as in the passage from the class $\GG$ of finite graphs to the class $\GG^<$ of finite ordered graphs). We examine in detail whether or not the generic order-expansion of a \fraisse class induces a different dividing-line from its precursor's, showing that if the precursor's generic theory is unstable, then the additional ordering adds nothing. This allows us to show that if $\KK$ is a Ramsey\emph{-expandable} \fraisse classes with unstable generic theory, then $\C_\KK$ is a collapse-of-indiscernibles dividing-line.

\item[Section \ref{Sec_Case_studies}:] By way of some case studies, we examine the structure of the $\cleq$-ordering and the set of irreducible dividing-lines. To start, we identify some theories that are maximum under $\cleq$. We then examine the classes $\HH_r$ ($0<r<\omega$) of hypergraphs, verifying that they are all indecomposable and that they induce a strict chain of irreducible dividing-lines. Along the way, we also show that certain natural perturbations of hypergraphs -- excluding cliques, adding additional irreflexive symmetric relations -- do not have any effect on the associated dividing-lines.

\item[Section \ref{Sec_Conjectures}:] To conclude the paper, we discuss some open questions about the structure of the $\cleq$-ordering and the irreducible classes generated by this ordering.  We also explore questions around the relationship between these classes and collapse-of-indiscernibles dividing-lines.
\end{description}

\subsection{Notation and conventions}

\begin{defn}
We write $\TT$ for the class of all complete theories with infinite models that eliminate imaginaries.
If left unspecified, the language of a theory $T\in\TT$ is $\L_T$, which may be, and usually is, many-sorted.
\end{defn}

\begin{conv}
When a theory $T$ arises that may not eliminate imaginaries, we freely identify $T$ with $T^\eq$, which of course eliminates imaginaries.
\end{conv}

\begin{defn}
Let $\L$ be a finite relational language (meaning that its signature $\sig(\L)$ consists of finitely many relation symbols and no function symbols or constant symbols). $\Fin(\L)$ denotes the class of all finite $\L$-structures.
A \fraisse class in $\L$ is a sub-class $\KK\subseteq\Fin(\L)$ satisfying the following:
\begin{description}
\item[Heredity (HP):] For all $A,B\in\Fin(\L)$, if $A\leq B$ and $B\in\KK$, then $A\in\KK$.
\item[Joint-embedding (JEP):] For any two $B_1,B_2\in\KK$, there are $C\in\KK$ and embeddings $B_1\to C$ and $B_2\to C$.
\item[Amalgamation (AP):] For any $A,B_1,B_2\in\KK$ and any embeddings $f_i:A\to B_i$ ($i=1,2$), there are $C\in\KK$ and embeddings $f'_i:B_i\to C$ ($i=1,2$) such that $f_1'{\circ}f_1 = f'_2{\circ}f_2$.
\end{description}
According to \cite{big-hodges}, a \fraisse class $\KK$ has a generic model (or \fraisse limit) $\A$ satisfying the following properties:
\begin{itemize}
\item $\KK$-universality: For every $B\in\KK$, there is an embedding $B\to\A$.
\item $\KK$-closedness: For every finite $X\subset A$, the induced substructure $\A\r X$ is in $\KK$.
\item Ultrahomogeneity (or $\KK$-homogeneity): For every $X\subset_\fin A$ and every embedding $f:\A\r X\to \A$, there is an automorphism $g\in Aut(\M)$ extending $f$.
\end{itemize}
Since we are working with only finite relational languages (for \fraisse classes), the generic theory $T_\KK = \Th(\A)$ of $\KK$ is always $\aleph_0$-categorical and eliminates quantifiers.

In general, a theory $T$ is {\em algebraically trivial} if $\acl^\M(A) = A$ whenever $\M\models T$ and $A\subseteq M$. A \fraisse class $\KK$ is called algebraically trivial just in case $T_\KK$ is algebraically trivial, and this condition can characterized by strengthened amalgamation conditions:
\begin{description}
\item[Disjoint-JEP:] For any two $B_1,B_2\in\KK$, there are $C\in\KK$ and embeddings $f_1:B_1\to C$ and $f_2:B_2\to C$ such that $f_1B_1\cap f_2B_2 = \emptyset$.
\item[Disjoint-AP:] For any $A,B_1,B_2\in\KK$ and any embeddings $f_i:A\to B_i$ ($i=1,2$), there are $C\in\KK$ and embeddings $f'_i:B_i\to C$ ($i=1,2$) such that $f_1'{\circ}f_1 = f'_2{\circ}f_2$ and $f'_1B_1\cap f'_2B_1 = f_1'f_1A = f'_2f_2A$.
\end{description}
That is, a \fraisse class $\KK$ is algebraically trivial if and only if it has disjoint-JEP and disjoint-AP.
\end{defn}

\begin{conv}
For \fraisse classes, we allow only finite relational signatures.
\end{conv}

\begin{defn}
Let $\KK_0,...,\KK_{n-1}$ be algebraically trivial \fraisse classes in languages $\L_0,...,\L_{n-1}$, respectively, such that $\sig(\L_i)\cap\sig(\L_j)=\emptyset$ whenever $i<j<n$. Let $\A_0,...,\A_{n-1}$ be the generic models of $\KK_0,...,\KK_{n-1}$, respectively.  Let $\Pi_i\L_i$ be the language with signature  $\bigcup_i\sig(\L_i)$. We define $\Pi_i\KK_i$ to be the class of finite $\Pi_i\L_i$-structures of the form $B=B_0{\times}{\cdots}{\times}B_{n-1}$, where $B_0\in\KK_0$,...,$B_{n-1}\in\KK_{n-1}$, with interpretations $$R^B=\left\{(\aa_0,...,\aa_{r-1})\in B^r:(a_{0,i},...,a_{r-1,i})\in R^{B_i}\right\}$$
for each $i<n$ and $R^{(r)}\in \sig(\L_i)$. In general, $\Pi_i\KK_i$ is not a \fraisse class (lacking quantifier elimination in its natural language), but it does have AP and JEP. Hence, $\Pi_i\KK_i$ has a well-defined generic model, and and one can show that this generic model is $\aleph_0$-categorical and algebraically trivial.

We have yet to define the ordering $\cleq$, but with the reader's indulgence, we use it now to define indecomposability for algebraically trivial \fraisse classes. Given an arbitrary algebraically trivial \fraisse class $\KK$ in $\L$ with generic model $\A$, a {\em factorization of $\KK$}
is a list $(\KK_0,...,\KK_{n-1})$ of algebraically trivial \fraisse classes $\KK_0,...,\KK_{n-1}$ ($n>1$) for which there is  an injection $u:A\to B$, where $\B$ is the generic model of $\Pi_i\KK_i$, such that for all $k$ and $\aa,\aa'\in A^k$, 
$$\qtp^\A(\aa)=\qtp^\A(\aa')\iff \tp^\B(u\aa) = \tp^\B(u\aa').$$
We then say that  {\em $\KK$ is indecomposable} if for any factorization $(\KK_0,...,\KK_{n-1})$ of $\KK$, there is an $i<n$ such that $T_\KK\cleq T_{\KK_i}$.  (See Definition \ref{Defn_TheOrdering}.)
\end{defn}

%
%
%

\newpage
\section{Comparing the local combinatorics of theories}\label{Sec_LC_ord}

In this section, we develop and formalize the idea of (positive) local combinatorics of first-order theories (Subsection \ref{Subsec_PLC}), showing that the local combinatorics of a theory on the finite level is captured by \fraisse classes ``embedded'' in models of that theory. After exposing that infrastructure, we use it in Subsection \ref{Subsec_CLEQ} to define an ordering (really, a quasi-ordering or pre-ordering) of first-order theories which will underpin all of the rest of our discussion of model-theoretic dividing-lines. We begin Subsection \ref{Subsec_IRREDs} with a discussion of what an ``irreducible'' dividing-line {\em ought} to be, and thereafter, we formalize these ideas under the heading of complete prime filter classes. In the same subsection, we state the main theorem of this section, Theorem \ref{Thm_prime_equals_fraisse}, in which we show that the irreducible dividing-lines, in our sense, are precisely the ones that are characterized by indecomposable \fraisse classes, and in Subsection \ref{Subsec_PROOF}, we prove the theorem. In a final subsection, we also show that our formulation of $\cleq$, and of local combinatorics generally, using only one-sorted \fraisse classes would not capture anything more if we had allowed multiple sorts.

\begin{conv}
For all of this section, we fix a countably infinite set $A$. It will serve as the universe of generic models of \fraisse classes involved in various theories $T\in\TT$.
\end{conv}

\subsection{Infrastructure}\label{Subsec_PLC}

We will understand the local combinatorics of a theory $T$ as a collection of functions from finite subsets of $A$ into models of $T$ -- see Definition \ref{Def_LC_functions}. (In order to avoid the ``geometric'' behavior of $T$, we will require that these functions have strong algebraic-triviality properties.) Such a function and a list of formulas of the language of $T$ will induce a structure on the domain of the function, and really, it is this structure that encodes some of the combinatorics of $T$ -- see Definition \ref{Def_LC_structures}.

\begin{defn}\label{Def_LC_functions}
Let $T\in\TT$. 
\begin{itemize}
\item We define $\FF(T)$ to be the set of injections $f:B\to \|\M\|$ where:
\begin{itemize}
\item $B\subset_\fin A$ and $\M\models T$.
\item For all $a,b\in B$, $f(a)$ and $f(b)$ are in the same sort of $\M$.
\item For every $B_0\subsetneq B$, $\tp(f[B]/f[B_0])$ is non-algebraic.
\end{itemize}
\item $\vec\FF(T)$ is the set of functions $F:A\to\|\M\|$, for some $\M \models T$, such that $F\r B\in\FF(T)$ for all $B\subset_\fin A$.
\end{itemize}
\end{defn}

\begin{conv}
Let $T\in\TT$. Consider a sequence  
$$\pphi = (\phi_0(x_0,...,x_{n_0-1}),...,\phi_{N-1}(x_0,...,x_{n_{N-1}-1}))$$
 of $\L_T$-formulas.
 Whenever we speak of sequences of formulas, we understand that the sequence is finite and all of the free variables range over  a single common sort which we call {\em the sort of $\pphi$}. Associated with $\pphi$, we have a language $\L_\pphi$ with signature $\left\{R_0^{(n_0)},...,R_{N-1}^{(n_{N-1})}\right\}$.
If $\pphi^1$ and $\pphi^2$ have the same length and coordinate-wise have the same arities, then the signatures of $\L_{\pphi^1}$ and $\L_{\pphi^2}$ are identical, so we identify $\L_{\pphi^1}$ and $\L_{\pphi^2}$ to compare structures in these languages.
\end{conv}

\begin{defn}\label{Def_LC_structures}
Let $T\in\TT$, and $\pphi$ be a sequence of formulas of $\L_T$. 
\begin{itemize}
\item $\FF_\pphi(T)$ is the set of all $f\in\FF(T)$ such that for each $a\in \dom(f)$, $f(a)$ is in the sort of $\pphi$. We define $\vec\FF_\pphi(T)$ similarly. In both cases, $f\in\FF_\pphi(T)$ or $F\in\vec\FF_\pphi(T)$ is said to be {\em compatible with $\pphi$}.

\item Let $f:B\to\|\M\|$ be in $\FF_\pphi(T)$. Then, we write $B_\pphi(f)$ for the $\L_\pphi$-structure with universe $B$ and interpretations 
$$R_i^{B_\pphi(f)} = \left\{\aa\in B^{n_i}:\M\models\phi_i(f\aa)\right\}.$$
We then define $\Age_\pphi(T) = \left\{B_\pphi(f):f\in\FF_\pphi(T)\right\}$ up to isomorphism.

\item For $F\in\vec\FF_\pphi(T)$, we define an $\L_\pphi$-structure $\A_\pphi(F)$ with universe $A$ and intepretations
$$R_i^{\A_\pphi(F)} = \left\{\aa\in A^{n_i}:\M\models\phi_i(f\aa)\right\}.$$
We then define $\Age_\pphi(F) = \Age(\A_\pphi(F))$. 
\item Let $F\in\vec\FF_\pphi(T)$. We say that $F$ is {\em $\pphi$-resolved}  if for any $B\subsetneq B'\subset_\fin A$, there are $B_0',...,B_i',...,\subset_\fin A$, such that $B'_i\cap B'_j= B$ whenever $i<j<\omega$, and $\tp_\pphi(F[B'_i]/F[B]) = \tp_\pphi(F[B']/F[B])$ for all $i<\omega$.

\item Finally, we define $\SS_\pphi(T) = \left\{\A_{\pphi}(F):F\in\vec\FF_\pphi(T),\textit{$\pphi$-resolved}\right\}$.
\end{itemize}
\end{defn}

\bigskip

We have now defined the various objects that encode the positive local combinatorics of a theory $T$. In the proposition below, we show that all of this is, essentially, explicably on the the level of algebraically trivial \fraisse classes ``embedded'' in models of $T$ as $\A_\pphi(F)$'s.

\begin{prop}\label{Prop_resolved}
Let $T\in\TT$, and let $\pphi = (\phi_0(x_0,...,x_{n_0-1}),...,\phi_{N-1})$ be a sequence of $\L_T$-formulas. Let $F\in\vec\FF_\pphi(T)$.
\begin{enumerate}
\item There is a $\pphi$-resolved $F'\in\vec\FF_\pphi(T)$ such that $\Age_\pphi(F)\subseteq\Age_\pphi(F')$.
\item  If $F$ is $\pphi$-resolved, then $\Age_\pphi(F)$ is an algebraically trivial \fraisse class 
and $\A_\pphi(F)$ is the \fraisse limit of $\Age_\pphi(F)$.
\end{enumerate}
\end{prop}
\begin{proof}
For Item 1, let $F$ be given. Let $\sim$ be the equivalence relation on $\vec\FF_\pphi(T)$ given by 
$$F_1\sim F_2\iff (\forall B\subset_\fin A) \left[ B_\pphi(F_1\r B) = B_\pphi(F_2\r B) \right].$$
Let $\XX = \vec\FF_\pphi(T)/{\sim}$. For $f\in\FF_\pphi(T)$, we define $[f]$ to be the set of classes $F'/{\sim}$ such that $B_\pphi(f) = B_\pphi(F'\r\dom(f))$; then 
$\tau_0 = \left\{ [f] : f \in \FF_\pphi(T) \right\}$ may be viewed as a base of clopen sets for a Stone topology on $\XX$.  We define two families of subsets of $\FF_\pphi(T)$ as follows:

\begin{itemize}
\item Let $B_0, B \in \Age_\pphi(F)$, $f_0$, $u_0$ be such that $B_0\leq B$, and let $f_0\in \FF_\pphi(T)$ such that $u_0:B_0\cong B_{\pphi}(f_0)$.  

Then, $f\in R^n(B,u_0,f_0)$ iff the following holds:

\textbf{If} $f_0\subseteq f$ and $f$ is compatible with $\pphi$, \textbf{then} there are $f_0\subseteq f_1,...,f_n\subseteq f$ and $u_0\subseteq u_1,...,u_n$ such that $u_i:B \cong B_{\pphi}(f_i)$ for each $1\leq i\leq n$ and $\img(f_i)\cap\img(f_j) = \img(f_0)$ for all $1\leq i<j\leq n$.

%

\item Let $B\in\Age_\pphi(F)$. Then $f\in R_F(B)$ iff $f$ is compatible with $\pphi$ and there is some $f_0\subseteq f$ such that $B_{\pphi}(f_0)\cong B$
\end{itemize}
By the definitions of $\FF_\pphi(T)$ and $\XX$, it is not difficult to check that 
$$\Gamma_F = \bigcup_{B} \left\{ [f] : f \in R_F(B) \right\} \cup \bigcup_{B, u_0, f_0, n} \left\{ [f] : f \in R^n(B, u_0, f_0) \right\}$$
is a countable family of dense-open sets, so as $\XX$ has the Baire property (because it is compact Hausdorff), the intersection $\bigcap\Gamma_F$ is non-empty. For any $F'/{\sim}$ in $\bigcap\Gamma_F$, $F'$ satisfies the requirements of Item 1.

\medskip

For Item 2: First, we observe that for any $F'\in\vec\FF_\pphi(T)$, $\Age_\pphi(F')$ has HP. Moreover, if $F$ is $\pphi$-resolved, then by definition, $\Age_\pphi(F)$ has disjoint-JEP and disjoint-AP -- so $\Age_\pphi(F)$ is an algebraically trivial \fraisse class. 
Since $F$ is $\pphi$-resolved, $\A_\pphi(F)$ is the generic model of $\Age_\pphi(F)$.
\end{proof}

\begin{cor}\label{Lemma_up_res}
Let $T\in\TT$, and let $\pphi$ be a sequence of $\L_T$-formulas.  Then, for any $f\in\FF_\pphi(T)$, there is a $\pphi$-resolved $F\in\vec\FF_\pphi(T)$ such that $B_{\pphi}(f)\in\Age_\pphi(F)$. 
\end{cor}
\begin{proof}
Given $f\in\FF_\pphi(T)$, we choose any $F_0\in\vec\FF_\pphi(T)$ such that $f\subset F_0$. Applying Proposition \ref{Prop_resolved}, we then obtain a $\pphi$-resolved $F\in\vec\FF_\pphi(T)$ such that $B_\pphi(f)\in \Age_\pphi(F_0)\subseteq \Age_\pphi(F)$.
\end{proof}

\begin{obs}\label{Obs_combine_pictures}
Let $T\in\TT$, and let $\pphi$ be a sequence of $\L_T$-formulas. Let $F_0,...,F_n,...\in\vec\FF_\phi(T)$. Then, there are a partition $\{A_n\}_{n<\omega}$ of $A$ into infinite classes, injections $u_n:A\to A_n$, and an $F\in\vec\FF_\phi(T)$ such that for each $n<\omega$, $\A_\pphi(F_n) = \A_\pphi(F\circ u_n)$.
\end{obs}
\begin{proof}
We choose the partition $\{A_n\}_n$ arbitrarily, and we construct $F$ by a routine compactness argument.
\end{proof}

\begin{obs}\label{Obs_countable_dense_subset}
The space $\XX_\pphi(T)$ (denoted $\XX$) in the proof of Proposition \ref{Prop_resolved} has a countable dense subset $W_\pphi(T)$ such that $F$ is $\pphi$-resolved for each $F/{\sim}$ in $W_\pphi(T)$.
\end{obs}

\subsection{The ordering $\cleq$ (and some more infrastructure)}\label{Subsec_CLEQ}

We have discussed and formalized our notion of local combinatorics of first-order theories, and now, we use these ideas to formulate an ordering $\cleq$ of $\TT$ that will allow us to compare theories based on their local combinatorics. Initially, we present a definition in which $T_1$ and $T_2$ are compared by way finite structures,  $B_\pphi(f)$'s, but in Proposition \ref{Prop_variants_of_CLEQ}, we demonstrate (unsurprisingly, given Proposition \ref{Prop_resolved}) that this is equivalent to comparing $T_1$ and $T_2$ based on (generic models of) algebraically trivial \fraisse classes.

\begin{defn}\label{Defn_TheOrdering}
Let $T_1,T_2\in\TT$. Then we assert $T_1\cleq T_2$ if for every finite sequence $\pphi^1 = (\phi^1_0,...,\phi^1_{N-1})$ of $\L_{T_1}$-formulas, there is  $\pphi^2 = (\phi^2_0,...,\phi^2_{N-1})$ in $\L_{T_2}$, coordinate-wise of the same arities,  such that 
$\Age_{\pphi^1}(T_1)\subseteq\Age_{\pphi^2}(T_2)$. 
\end{defn}

\begin{obs}
Let $T_1,T_2\in\TT$. If $T_1$ is interpretable in $T_2$, then $T_1\cleq T_2$.
\end{obs}

\begin{prop}\label{Prop_variants_of_CLEQ}
Let $T_1,T_2\in\TT$. The following are equivalent:
\begin{enumerate}
\item $T_1\cleq T_2$.
\item For every finite sequence $\pphi^1 = (\phi^1_0,...,\phi^1_{N-1})$ of $\L_{T_1}$-formulas, there is  $\pphi^2 = (\phi^2_0,...,\phi^2_{N-1})$ in $\L_{T_2}$, coordinate-wise of the same arities,  such that for every $\pphi^1$-resolved $F_1\in\vec\FF_{\pphi^1}(T_1)$, there is a $\pphi^2$-resolved $F_2\in\vec\FF_{\pphi^2}(T_2)$ such that $\A_{\pphi^1}(F_1)=\A_{\pphi^2}(F_2)$. 

\item For every finite sequence $\pphi^1 = (\phi^1_0,...,\phi^1_{N-1})$ of $\L_{T_1}$-formulas, there is  $\pphi^2 = (\phi^2_0,...,\phi^2_{N-1})$ in $\L_{T_2}$, coordinate-wise of the same arities,  such that $\SS_{\pphi^1}(T_1)\subseteq \SS_{\pphi^2}(T_2)$. 
\end{enumerate}
\end{prop}
\begin{proof}
2$\iff$3 is by definition of $\SS_\pphi(T)$. Let us write $T_1\cleq'T_2$ to mean that the condition expressed in item 2 holds. 

One the one hand, suppose $T_1\cleq'T_2$, and let $\pphi^1 = (\phi^1_0,...,\phi^1_{N-1})$, a sequence of $\L_{T_1}$-formulas, be given. Since $T_1\cleq'T_2$, let $\pphi^2$ be the promised sequence of $\L_{T_2}$-formulas. To show that $\Age_{\pphi^1}(T_1)\subseteq\Age_{\pphi^2}(T_2)$, let $f_1\in\FF_{\pphi^1}(T_1)$ be given. By Corollary \ref{Lemma_up_res}, we obtain a $\pphi^1$-resolved $F_1\in\vec\FF_{\pphi^1}(T_1)$ such that $B_{\pphi^1}(f_1)\in\Age_{\pphi^1}(F_1)$. By our choice of $\pphi^2$, then, there is a $\pphi^2$-resolved $F_2\in\vec\FF_{\pphi^2}(T_2)$ such that $\A_{\pphi^1}(F_1) = \A_{\pphi_2}(F_2)$. 
We have $B_{\pphi^1}(f_1)\in\Age_{\pphi^2}(F_2)\subseteq\Age_{\pphi^2}(T_2)$. We have shown that $T_1\cleq T_2$, which proves 2,3$\implies$1.

For 1$\implies$2,3, suppose $T_1\cleq T_2$. Again, let $\pphi^1 = (\phi^1_0,...,\phi^1_{N-1})$, a sequence of $\L_{T_1}$-formulas, be given. Let $\pphi^2$ be the sequence of $\L_{T_2}$-formulas promised by $T_1\cleq T_2$. Suppose $F_1\in\vec\FF_{\pphi^1}(T_1)$ is $\pphi^1$-resolved.

Let $a_0,a_1,...,a_n,...$ be an enumeration of $A$, and for each $n<\omega$, let $f_n = F_1\r\left\{a_0,...,a_{n-1}\right\}$. Since $T_1\cleq T_2$, for each $n$, we obtain $f'_n\in\FF_{\pphi^2}(T_2)$ with domain $\{a_0,...,a_{n-1}\}$ such that $B_{\pphi^1}(f_n) = B_{\pphi^2}(f_n')$. 
By definition of $\FF_{\pphi^2}(T_2)$, we can also ensure that $f_n'\subset f_{n+1}'$ for all $n<\omega$, so $F_2 = \bigcup_nf'_n$ is in $\vec\FF_{\pphi^2}(T_2)$. It is not difficult to verify that $F_2$ is $\pphi^2$-resolved and that $\A_{\pphi^1}(F_1) = \A_{\pphi^2}(F_2)$. 
Thus, $T_1\cleq'T_2$ -- as desired.
\end{proof}

%


Of course, generic models and generic theories of algebraically trivial \fraisse classes will play key role in our work later in this paper. We make one more convenient definition and two observations about $\cleq$ as it pertains to generic theories of \fraisse classes. These two observations -- \ref{Obs_AF_CLEQ_1} and \ref{Obs_witnessing} -- will be used repeatedly in the sequel, often without comment. 

\begin{defn}
For each theory $T\in\TT$, we define
$$Q_T = \left\{\Th(\A_{\pphi}(F))^\eq:\begin{array}{l}
\textnormal{$\pphi = (\phi_0,...,\phi_{N-1})$ in $\L_{T}$,}\\
F\in\vec\FF_\pphi(T) \textnormal{ $\pphi$-resolved}, \\
\end{array}\right\}$$
Later on, it will also be convenient to work with the following sub-class of theories:
$$\tilde\TT := \left\{T_\KK^\eq\,:\,\textnormal{$\KK$ is an alg. trivial \fraisse class}\right\}.$$
\end{defn}

\begin{obs}\label{Obs_AF_CLEQ_1}
In the definition of $Q_T$, each $F$ is $\pphi$-resolved, so by Proposition \ref{Prop_resolved}, $\A_\pphi(F)$ is the \fraisse limit of $\Age_\pphi(F)$, an algebraically trivial \fraisse class.  
Hence, $Q_T \subseteq \tilde\TT$.  Notice further that $T_0 \cleq T$ for all $T_0 \in Q_T$.  
\end{obs}

%
%

\begin{obs}\label{Obs_witnessing}
Let $\KK$ be an algebraically trivial \fraisse class in a finite relational language $\L$, respectively. Then, for any complete 1-sorted theory $T$, the following are equivalent:
\begin{enumerate}
\item $T_{\KK}\cleq T$.
\item There are $0<m<\omega$, formulas $\phi_R(\xx_0,...,\xx_{r-1})\in\L_T$ for each $R^{(r)}\in\sig(\L)$ (where each $\xx_j$ is a non-repeating $m$-tuple of variables), and $F\in\vec\FF_\pphi(T)$ such that $\KK=\Age_\pphi(F)$. 
\end{enumerate}
(The number $m$ is then said {\em to witness $T_{\KK}\cleq T$}.) In particular, for algebraically trivial \fraisse classes $\KK_1$, $\KK_2$ in languages $\L_1,\L_2$, with generic models $\A_1,\A_2$, respectively, the following are equivalent:
\begin{enumerate}
\item $T_{\KK_1}\cleq T_{\KK_2}$.
\item There are $0<m<\omega$, an injection $u:A_1\to A_2^m$, and quantifier-free formulas $\theta_R(\xx_0,...,\xx_{r-1})$ of $\L_2$ ($|\xx_i|=m$, $R^{(r)}\in\sig(\L_1)$) such that for all $R^{(r)}\in\sig(\L_1)$ and $a_0,...,a_{r-1}\in A_1$,
$$\A_1\models R(a_0,...,a_{r-1})\iff \A_2\models\theta_R(u(a_0),...,u(a_{r-1})).$$
\end{enumerate}
\end{obs}

\subsection{Irreducible model-theoretic dividing-lines}\label{Subsec_IRREDs}

\subsubsection{Discussion: What is an ``irreducible'' dividing-line?}\label{Subsec_what_is_a_div-line?}

As we have already discussed at some length, a model-theoretic dividing-line amounts to a partition of $\TT$ into two sub-classes -- a sub-class $\C$ of ``wild'' theories (unstable, IP, unsimple,...) and a complementary class $N\C:=\TT\setminus\C$ of ``tame'' theories. Although we prefer to work with theories from $N\C$ in practice, we can characterize a dividing-line purely in terms of the ``wild'' class $\C$. In principle, $\C$ could be any sub-class of $\TT$, but we have to demand more from $\C$ if $N\C$ is to have any practical value. This already suggests the most primitive requirement we make on irreducible dividing-lines (relative to any ordering $\leq$ of $\TT$):
\begin{enumerate}
\item {\em Existence}: $\C$ is not empty.
\item {\em Upward-closure}: If $T_1\leq T_2$ and $T_1\in\C$, then $T_2\in\C$.
\end{enumerate}
If we are to call $\C$ an irreducible dividing-line, there are several ``indivisibility'' or ``non-transience'' requirements that seem unavoidable:
\begin{enumerate}
\item[3.] Consider a theory $T$ obtained as a ``disjoint union'' of a family of theories $\left\{T_i\right\}_{i\in I}$; say, $T$ has a family of sorts for each $T_i$, and these sorts are orthogonal in $T$. For the sake of irreducibility, if $T\in\C$, then we should expect that at least one of the $T_i$'s is in $\C$. Otherwise membership in $\C$ would appear to depend on two or more phenomena that occur (or not) independently of each other.

That is to say, $\C$ should be {\em prime}.

\item[4.] Consider a finite family of theories $T_0,...,T_{n-1}\in\C$. It would be very strange to require that {\em every} theory $T$ that lies $\leq$-below each of $T_0,...,T_{n-1}$ to be in $\C$; if $\leq$ is in any way natural, this would presumably place the theory of an infinite set in $\C$. However, for the sake of irreducibility, we should expect the fact that $T_0,...,T_{n-1}$ are all in $\C$ to have a single common explanation. That is, there should be {\em some} $T$ in $\C$ and lies $\leq$-below  each of $T_0,...,T_{n-1}$.

Thus, $\C$ should have some sort of {\em ``finite intersection property''}.

\item[5.] Consider a descending chain $T_0 \geq T_1 \geq T_2 \geq \cdots$ of members of $\C$. Again, we should expect the fact that $T_0, T_1, ...$ are all in $\C$ to have a single common explanation, not infinitely many different explanations that, for no particular reason, happened upon a $\leq$-chain. Thus, we expect that this common explanation is a witnessed through $\leq$ by a theory -- that is, there should be some $T\in\C$ that lives $\leq$-below all $T_i$'s.

Therefore, $\C$ should have some sort of {\em ``completeness''} property.
\end{enumerate}
If we accept these strictures for a definition of ``irreducible'' dividing-line relative to an ordering $\leq$ of $\TT$, then our definition in the next subsection is forced on us. If we settle on this (or any) definition of irreducible dividing-line, then it is reasonably natural to ask if irreducible dividing-lines admit ``characterizing objects,'' and we address this question in Theorem \ref{Thm_prime_equals_fraisse}.

\subsubsection{Irreducibles: Complete prime filter classes}

Based on our discussion in Subsection \ref{Subsec_what_is_a_div-line?}, we now formalize the notion of an irreducible dividing-line relative to our ordering $\cleq$ of $\TT$; this formalization is given in Definition \ref{Def_irred_div-line} below. (Definition \ref{Def_cones}, which proceeds it, just establishes some helpful notation.) Given the definition of $\cleq$, it is probably not surprising that classes $\C_\KK\subset\TT$, defined from \fraisse classes $\KK$, will play an important role in the development, so we make these $\C_\KK$'s formal in Definition \ref{Def_C_K}. Finally, in Theorem \ref{Thm_prime_equals_fraisse}, we state the main result of this section, identifying irreducible dividing-lines with classes of theories defined from indecomposable \fraisse classes.

\begin{defn}\label{Def_cones}
For a set $S\subset\TT$, we define
$${\downarrow}S = \left\{T\in\TT:(\forall T_1\in S)\,  T\cleq T_1\right\}$$
$${\uparrow}S = \left\{T\in\TT:(\forall T_0\in S)\, T_0\cleq T\right\}$$
which are the lower- and upper-cones of $S$.
\end{defn}

\begin{defn}\label{Def_irred_div-line}
Let $\C\subseteq\TT$. We say that $\C$ is {\em irreducible} (or less succinctly, is a {\em complete prime filter class}) if the following hold:
\begin{itemize}
\item (Existence) $\C$ is non-empty.
\item (Filter properties) For any  $T_0,T_2,...,T_{n-1}\in\TT$ ($1<n<\omega$):
\begin{itemize}
\item If $T_0\in\C$ and $T_0\cleq T_1$, then $T_1\in\C$.
\item  If $T_0,T_1,...,T_{n-1}\in\C$, then $\C\cap{\downarrow}\{T_0,...,T_{n-1}\}\neq\emptyset$.
\end{itemize} 
\item (Completeness) If 
$(T_i)_{i\in I}$ is a non-empty $\cleq$-chain of theories in $\C$ (i.e. $I = (I,<)$ is a non-empty linear order, and for all $i,j\in I$, $i<j\implies T_i\cleq T_j$), then  $\C\cap {\downarrow} \left\{T_i\right\}_{i\in I}\neq\emptyset$.

\item (Primality) For any set $S\subset \TT$, if ${\uparrow} S\subseteq\C$, then $\C\cap S\neq\emptyset$.
\end{itemize}
\end{defn}

\begin{rem}
Our completeness axiom is definitely stronger than required: In fact, the following {\em $(2^{\aleph_0})^+$-completeness condition}  would suffice:
\begin{quote}
For any positive ordinal $\alpha < (2^{\aleph_0})^+$, for any descending $\cleq$-chain $(T_i)_{i<\alpha}$ (i.e. $i<j<\alpha\implies T_j\cleq T_i$)
of members of $\C$ of length $\alpha$, $\C\cap {\downarrow} \left\{T_i\right\}_{i<\alpha}\neq\emptyset$.
\end{quote}
We conjecture that even this $(2^{\aleph_0})^+$-completeness condition is stronger than necessary.
\end{rem}

\begin{defn}\label{Def_C_K}
Let $\KK$ be an algebraically trivial \fraisse class. Then $\C_\KK$ is the class of theories $T\in\TT$ for which there are a sequence of formulas $\pphi = \left(\phi_R(x_0,...,x_{r-1})\right)_{R^{(r)}\in\sig(\L_\KK)}$ in $\L_T$ and a $\pphi$-resolved $F\in\vec\FF_\pphi(T)$ such that $\Age_\pphi(F)=\KK$. 
\noindent In other words,
\[
 \C_\KK = \{ T \in \TT : T_\KK \cleq T \}.
\]
We note that if $\KK_1$ and $\KK_2$ are two algebraically trivial \fraisse classes, then
\[
 T_{\KK_1} \cleq T_{\KK_2} \iff \C_{\KK_2} \subseteq \C_{\KK_1}.
\]
\end{defn}

As promised, we now state the main result of this section, which says that irreducible dividing-lines (relative to $\cleq$) having fairly concrete ``characterizing objects,'' namely indecomposable \fraisse classes, and that any class defined from one of these is, indeed, the ``wild'' class of an irreducible dividing-line. Of course, this reduces the project to identifying the indecomposable \fraisse classes, and we take small steps in this project in the ensuing sections of the paper. The proof of Theorem \ref{Thm_prime_equals_fraisse} is given in the next subsection.

\begin{thm}\label{Thm_prime_equals_fraisse}
Let $\C$ be a non-empty class of theories. The following are equivalent:
\begin{enumerate}
\item $\C$ is irreducible.
\item $\C=\C_\KK$ for some indecomposable algebraically trivial \fraisse class $\KK$.
\end{enumerate}
\end{thm}

To conclude this subsection, we observe that the identification of irreducible dividing-lines with certain kinds of \fraisse classes yields an easy upper bound on the number of such dividing-lines. 
\begin{cor}\label{Cor_continuum_bound}
There are no more than $2^{\aleph_0}$ indecomposable algebraically trivial \fraisse classes (in finite relational langauges). So by Theorem \ref{Thm_prime_equals_fraisse}, if $\mathscr{F}$ denotes the family of all irreducible dividing-lines (complete prime filter classes), then $\big|\mathscr{F}\big|\leq 2^{\aleph_0}$.
\end{cor}
\begin{proof}
Let $S$ be the set of functions $s:\omega\to\omega$ of finite support. For each $s\in S$, let $\sig(\L_s)$ be the signature with relation symbols $R_i^{(s(i))}$ for each $i\in \mathrm{supp}(s)$, and let $\K_s$ be the set of all algebraically trivial \fraisse classes of finite $\L_s$-structures. We define $\K_s^*$ to be the set of all indecomposable algebraically trivial \fraisse classes of finite $\L_s$-structures. Then 
$$\big|\mathscr{F}\big|\leq\Big|\,{\bigcup}_{s\in S}\K_s^*\Big|\leq \aleph_0\cdot 2^{\aleph_0} = 2^{\aleph_0}.$$
\end{proof}

\subsection{Proof of Theorem \ref{Thm_prime_equals_fraisse}}\label{Subsec_PROOF}

We now turn to the proof of Theorem \ref{Thm_prime_equals_fraisse}, which of course has two directions. 
The proof of 2$\implies$1 in Theorem \ref{Thm_prime_equals_fraisse} is quite short, so we give it immediately in the form of Proposition \ref{Prop_indec_yields_irred}.
The proof of 1$\implies$2  (Proposition \ref{Prop_prime_implies_fraisse}, below) is somewhat more involved.

\begin{prop}\label{Prop_indec_yields_irred}
Let $\KK$ be an indecomposable algebraically trivial \fraisse class of finite relational structures. Then
$\C_\KK$ is irreducible. 
\end{prop}
\begin{proof}
Let $\L$ be the language of $\KK$, say $\sig(\L)=\left\{R_0^{(r_0)},...,R_{k-1}^{(r_{k-1})}\right\}$. Only the primality of $\C_\KK$ is not altogether obvious. To prove primality, let $S\subset\TT$ be a set of theories. Without loss of generality, we assume that for each $T\in\TT$, $\L_T$ is purely relational. We define a language $\L_S$ as follows:
\begin{itemize}
\item For each $T\in\TT$, for each sort $\sortX$ of $\L_T$, $\L_S$ has a sort $\sortY_{T:\sortX}$
\item For each $T\in\TT$, for each relation symbol $R\subseteq\sortX_0\times\cdots\times\sortX_{n-1}$ of $\L_T$, $\L_S$ has a relation symbol $R_T\subseteq \sortY_{T:\sortX_0}\times\cdots\times \sortY_{T:\sortX_{n-1}}$
\end{itemize}
For an $\L_S$-structure $\M$, we take $\M_T$ to denote the restriction/reduct of $\M$ to the sorts $\sortY_{T:\sortX}$ and symbols $R_T$ associated with $T$. We define $T_S$ to be the theory of $\L_S$-structures $\M$ such that $\M_T\models T$ for every $T\in S$. For $S_0\subseteq S$, we define $\M_{S_0}$ similarly.

It is not hard to see that $T_S\in{\uparrow}S$. It is routine to verify that $T_S$ is complete, that up to the obvious translations of formulas, $\bigcup_{T\in S}\L_T$ is an elimination set of $T_S$, and that $T_S$ eliminates imaginaries. Finally, we observe that for pairwise distinct $T_0,...,T_{n-1}\in S$ and $\M\models T_S$,  $\M_{T_0},...,\M_{T_{n-1}}$ are orthogonal in the sense that any 0-definable set $D$ of $\M_{\{T_i\}_i}$ is equal to a boolean combination of sets of the form $D_0\times\cdots\times D_{n-1}$, where $D_i$ is a 0-definable set of $\M_{T_i}$ for each $i<n$.

Now, suppose that ${\uparrow}S\subseteq\C_\KK$ -- so of course, $T_S\in\C_\KK$. We must show that there is some $T\in S$ such that $T\in\C_\KK$. Since $T_S\in \C_\KK$, there are $0<m<\omega$, formulas $\phi_i(\xx_0,...,\xx_{r_i-1})\in\L_S$ for each $i<k$ (where each $\xx_j$ is a non-repeating $m$-tuple of variables, say in the sort $\sortY_{T_0:\sortX_0} ... \sortY_{T_{m-1}:\sortX_{m-1}}$), and $F\in\vec\FF_\pphi(T_S)$ such that $\KK=\Age_\pphi(F)$.
For each $j < m$ and $i < k$, let $\phi^j_i$ be the reduct of $\phi_i$ to $\sortY_{T_j:\sortX_j}$, let $F_j$ be the restriction of $F$ to $\sortY_{T_j:\sortX_j}$, and let $\KK_j = \Age_{\pphi^j}(F_j)$.  Then, $(\KK_0, ..., \KK_{m-1})$ is a factorization of $\KK$.
As $\KK$ is indecomposable, it follows that $T_\KK\cleq T_{\KK_j}\cleq T_j$ for some $j<m$, and then $T_j\in\C_\KK$ -- as required.
\end{proof}

\bigskip

\begin{prop}\label{Prop_prime_implies_fraisse}
If $\C$ is irreducible, then there is an indecomposable algebraically trivial \fraisse class $\KK$  such that $\C=\C_\KK$.
\end{prop}

For the rest of this subsection (the proof of Proposition \ref{Prop_prime_implies_fraisse}), we a fix a complete prime filter class $\C$.

The first important step in the proof of Proposition \ref{Prop_prime_implies_fraisse} is to identify the role of the \fraisse class $\KK$ in $\C$ in terms of $\cleq$. Unsurprisingly, we find that $\KK$ is chosen so that $T_\KK$ is the $\cleq$-minimum element of $\tilde\C = \tilde\TT\cap\C$, and we then demonstrate (Lemma \ref{Lemma_minimum_yields_indecomp}) being minimum for an irreducible class $\C$ is sufficient for indecomposability.

\begin{obs}
Let $\KK_0,...,\KK_{n-1}$ be algebraically trivial \fraisse classes, and let $\B$ be the generic model of $\Pi_i\KK_i$. For $T\in\TT$, if $T_{\KK_i}\cleq T$ for each $i<n$, then $\Th(\B)\cleq T$.
\end{obs}

\begin{lemma}\label{Lemma_minimum_yields_indecomp}
Let $\KK$ be an algebraically trivial \fraisse class such that $T_\KK$ is $\cleq$-minimum in 
$\tilde\C = \tilde\TT\cap\C$. Then $\KK$ is indecomposable.
\end{lemma}
\begin{proof}
Let $\A$ be the generic model of $\KK$. Suppose $(\KK_0,...,\KK_{n-1})$ is a factorization of $\KK$ via an injection $u:A\to B$, where $\B$ is the generic model of $\Pi_i\KK_i$. Obviously, $T_\KK\cleq \Th(\B)$.

Let $S = \left\{T_{\KK_i}\right\}_{i<n}$. We have observed that if $T_{\KK_i}\cleq T$ for each $i<n$, then $\Th(\B)\cleq T$. Thus, for any $T\in{\uparrow}S$, we have $T_\KK\cleq \Th(\B)\cleq T$, so ${\uparrow}S\subseteq\C$. Since $\C$ is prime, it follows that $T_{\KK_i}\in\C$ for some $i<n$. Since $T_\KK$ is $\cleq$-minimum in $\tilde\C$, we find that $T_\KK\cleq T_{\KK_i}$ -- as required.
\end{proof}

By Lemma \ref{Lemma_minimum_yields_indecomp}, we now know that in order to prove Proposition \ref{Prop_prime_implies_fraisse}, it is sufficient just to prove that $\tilde\C$ has a $\cleq$-minimum element, and that is what we do in the rest of the proof. This amounts to demonstrating, first, that a $\cleq$-minim\emph{al} element of $\tilde\C$ is already $\cleq$-minim\emph{um}, and second, that $\tilde\C$ must indeed have $\cleq$-minim\emph{al} element. The first project accounts for Lemmas \ref{Cor_support_by_fraisse}, \ref{Cor_tildeC_is_filter}, and \ref{Lemma_minimal_to_minimum}. The second part accounts for Lemma \ref{Lemma_have_minimal} and Corollary \ref{Cor_tildeCminimal}.

\begin{lemma}\label{Cor_support_by_fraisse}
For any $T\in\TT$, $T\in\C$ if and only if $Q_T\cap\C\neq\emptyset$. 
\end{lemma}
\begin{proof}
Clearly, if $Q_T\cap\C$ is non-empty, then $T\in\C$, so we just need to deal with the converse. Suppose $T\in\C$.


Let $\pphi$ be a sequence of $\L_T$-formulas. By Observation \ref{Obs_countable_dense_subset}, $\XX_\pphi(T)$, the space representing $\vec\FF_\phi(T)$ up to ``isomorphism'' from the proof of Proposition \ref{Prop_resolved}, has a countable dense subset $W_\pphi(T)$ such that $F$ is $\pphi$-resolved whenever $F/{\sim}\in W_\pphi(T)$. By Observation \ref{Obs_combine_pictures} and Proposition \ref{Prop_resolved}, there is a $\pphi$-resolved $F_\pphi\in\vec\FF_\pphi(T)$ such that for every $F/{\sim}\in W_\pphi(T)$, $\A_\pphi(F)$ embeds into $\A_\pphi(F_\pphi)$. One easily verifies, then, that for every $\pphi$-resolved $F\in\vec\FF_\pphi(T)$, $\A_\pphi(F)$ embeds into $\A_\pphi(F_\pphi)$. Now, we observe that for an arbitrary theory $T'\in\TT$, 
$$T'\in{\uparrow} Q_T\,\,\Longrightarrow\,\, (\forall \pphi\textnormal{ of }\L_T)\,Th(\A_\pphi(F_\pphi))\cleq T'\,\,\Longrightarrow\,\, T\cleq T'\,\,\Longrightarrow\,\, T'\in\C.$$
More succinctly, we have shown that ${\uparrow}Q_T\subseteq\C$. Since $\C$ is prime, 
$Q_T\cap\C\neq\emptyset$, as desired.
 \end{proof}


\begin{lemma}\label{Cor_tildeC_is_filter}
The sub-class
$\tilde\C = \C\cap\tilde\TT$
is a complete filter class (but not necessarily prime) relative to $\tilde\TT$. 
\end{lemma}
\begin{proof}
Since $\C$ is non-empty, we may choose $T\in\C$. Clearly,  $Q_T\cap\C\subseteq\tilde\C$, so as $Q_T\cap\C$ is non-empty, $\tilde\C\neq\emptyset$ as well.

For the first filter requirement, let $T_1,T_2\in\tilde\TT$, and suppose that $T_1\in\tilde\C$ and $T_1\cleq T_2$. Since $\C$ is a filter class, $T_2\in\C$, so $T_2\in\C\cap\tilde\TT=\tilde\C$. For the second filter requirement, let $T_0,...,T_{n-1}\in\tilde\C$. We claim that $\tilde\C\cap {\downarrow}\left\{T_0,...,T_{n-1}\right\}$ is non-empty. Since $\C$ is a filter class, let $T'\in\C\cap {\downarrow}\left\{T_0,...,T_{n-1}\right\}$. By Lemma \ref{Cor_support_by_fraisse}, $Q_{T'}\cap\C$ is non-empty, so let $T'_0\in Q_{T'}\cap\C$. Since $Q_{T'}\subseteq\tilde\TT$, we have $T'_0\in\tilde\C\cap {\downarrow}\left\{T_0,...,T_{n-1}\right\}$, as required.

For the completeness of $\tilde\C$, let $\kappa$ be a positive ordinal, and let $(T_i)_{i<\kappa}$ be a descending $\cleq$-chain of members of $\tilde\C$. By the completeness of $\C$, 
 there is a theory $T\in\C$ that is $\cleq$-below all of the $T_i$'s ($i<\kappa$). By Lemma \ref{Cor_support_by_fraisse} again, $Q_T\cap\C$ is non-empty, and any $T^*\in Q_T\cap\C$ is in $\tilde\C$ and also $\cleq$-below all of the $T_i$'s.
\end{proof}

\begin{lemma}\label{Lemma_minimal_to_minimum}
If $\tilde\C$ has at least one $\cleq$-minimal element, then it has a $\cleq$-minim\emph{um} element.
\end{lemma}
\begin{proof}
Let $\KK_0$ be an algebraically trivial \fraisse class such that $T_{\KK_0}$ is a $\cleq$-minimal element of $\tilde\C$. If $\KK_0$ is not $\cleq$-minim\emph{um}, then there is some $\KK_1$ such that $T_{\KK_1}\in\tilde\C$ such that $T_{\KK_0}\not\cleq T_{\KK_1}$. Since $\tilde\C$ is a filter class (specifically, the second requirement), there is an algebraically trivial \fraisse class $\KK^*$ such that $T_{\KK^*}$ is in $\tilde\C$, $T_{\KK^*}\cleq T_{\KK_0}$, and $T_{\KK^*}\cleq T_{\KK_1}$. Since $T_{\KK_0}$ is $\cleq$-minimal, we have $T_{\KK_0}\cleq T_{\KK^*}\cleq T_{\KK_1}$ -- a contradiction. Thus, $T_{\KK_0}$ is in fact a $\cleq$-minimum element of $\tilde\C$.
\end{proof}

\medskip

We have verified that a $\cleq$-minimal element of $\tilde\C$ is already $\cleq$-minimum, and now we need to show that $\tilde\C$ does indeed have $\cleq$-minimal element. The proof of this fact goes through showing that the lack of a $\cleq$-minimal element implies the existence of long descending chains in $\tilde\C$, which violates the following easy observation.


\begin{obs}\label{Obs_chain_bound}
Since there are, at most, $2^{\aleph_0}$-many algebraically trivial \fraisse classes (see the proof of Corollary \ref{Cor_continuum_bound}), $|\tilde\TT| \le 2^{\aleph_0}$.  Therefore, $\tilde\TT$ contains no strictly descending $\cleq$-chains of length greater than $2^{\aleph_0}$.  
\end{obs}


\begin{lemma}\label{Lemma_have_minimal}
If $\tilde\C$ does not have a $\cleq$-minimal element, then it contains a strictly descending $\cleq$-chain of length $(2^{\aleph_0})^+$.
\end{lemma}
\begin{proof}
At each stage $s<(2^{\aleph_0})^+$ of the following process, we will have a strictly descending $\cleq$-chain, so that 
$$k<\ell\leq s\implies T_{\KK_k}\rhd T_{\KK_\ell}.$$
\begin{itemize}
\item Choose $\KK_0$ arbitrarily subject to $T_{\KK_0}\in\tilde\C$.
\item At a successor stage $i+1$, since $\tilde\C$ does not have any $\cleq$-minimal elements, $T_{\KK_{i}}$ is not $\cleq$-minimal in $\tilde\C$, and we may choose $\KK_{i+1}$ such that $T_{\KK_{i+1}}\lhd T_{\KK_{i}}$ and $T_{\KK_{i+1}}\in\tilde\C$.

\item At a limit stage $\ell$, we are faced with a chain 
$$T_{\KK_0}\rhd\cdots\rhd T_{\KK_{i}}\rhd T_{\KK_{i+1}}\rhd\cdots$$
in $\tilde\C$. 
Since $\tilde\C$ is complete, we may choose $\KK_{\ell}$ such that $T_{\KK_\ell}\in\tilde\C$ and $T_{\KK_{\ell}}\cleq T_{\KK_{i}}$ for all $i<\ell$. We observe that if $T_{\KK_{i}}\cleq T_{\KK_{\ell}}$ for some $i<\ell$, then we would find 
$$T_{\KK_i}\cleq T_{\KK_\ell}\cleq T_{\KK_{i+1}}\lhd T_{\KK_i}$$
which is impossible; hence $T_{\KK_{\ell}}\lhd T_{\KK_{i}}$ for all $i<\ell$.
\end{itemize}
\end{proof}

\begin{cor}\label{Cor_tildeCminimal}
$\tilde\C$ has a $\cleq$-minimum element.
\end{cor}
\begin{proof}
By Observation \ref{Obs_chain_bound} and Lemma \ref{Lemma_have_minimal}, $\tilde\C$ has a $\cleq$-minimal element, say $T_\KK$, and by Lemma \ref{Lemma_minimal_to_minimum}, $T_\KK$ is $\cleq$-minimum in $\tilde\C$.
\end{proof}

\begin{proof}[Proof of Proposition \ref{Prop_prime_implies_fraisse}]
 Let $\KK$ be the algebraically trivial \fraisse class such that $T_\KK$ is $\cleq$-minimum in $\tilde\C$.  Fix any $T \in \C$.  Fix $T_0 \in Q_T \cap \C \subseteq \tilde\C$.  Then, $T_\KK \cleq T_0 \cleq T$, so $T \in \C_\KK$.  Conversely, fix $T \in \C_\KK$.  Then, $T_\KK \cleq T$ and $\C$ is upward closed, so $T \in \C$.  Therefore, $\C = \C_\KK$.
\end{proof}

This completes the proof 
Theorem \ref{Thm_prime_equals_fraisse}.

\subsection{Reduction to one sort}\label{Subsec_one_sort}

The reader may have noticed that, seemingly arbitrarily, our ordering $\cleq$ accommodates only one-sorted \fraisse classes, or one-sorted local combinatorics of theories. In this subsection, we justify this, showing that in fact additional (but finitely many) sorts yield no additional power over our one-sorted formulation. 

\begin{thm}\label{Thm_one_sort}
Let $\KK$ be an algebraically trivial \fraisse class in a $p$-sorted language $\L$ (sorts $\sortS_0,...,\sortS_{p-1}$) with generic model $\B$, and let $T\in\TT$. Then there is an algebraically trivial \fraisse class $\tilde\KK$ in a 1-sorted language such that the following are equivalent:
\begin{enumerate}
\item There are saturated $\M\models T$, $X_0,...,X_{p-1}$ definable sets of $\M$, $f_i:\sortS_i(\B)\to X_i$ injections ($i<p$), and for each $R\subseteq \sortS_{i_0}\times\cdots\times\sortS_{i_{r-1}}$ in $\sig(\L)$,  a formula $\phi_R(x_{0},...,x_{{r-1}})$ of $\L_T$ such that 
$$\B\models R(b_0,...,b_{r-1})\iff\M\models\phi_R(f_{i_0}(b_0),...,f_{i_{r-1}}(b_{r-1})).$$
\item $T\in\C_{\tilde\KK}$.
\end{enumerate}
\end{thm}

The proof of Theorem \ref{Thm_one_sort}, of course, requires that we define a \fraisse class $\tilde\KK$ and determine how it is related to $\KK$ itself. Lemma \ref{Lemma_transfer_1sort} is a transfer result matching members of $\KK$ directly with members of $\tilde\KK$

\begin{defn}
Let $\L$ be a $p$-sorted finite relational language with sorts $\sortS_0,...,\sortS_{p-1}$, and let $\KK$ be a \fraisse class of finite $\L$-structures. Then, let $\L_\KK$ be the one-sorted language with relation symbols $R_q^{(r)}$ for each irreflexive quantifier-free-complete type $q\in S_r^\qf(T_\KK)$, $r\leq\ari(\L)$.

\begin{itemize}
\item For each $B\in\KK$, we define an $\L_\KK$-structure $A_B$ and a family of maps $u_i^B:\sortS_i(B)\to A_B$ as follows:
\begin{itemize}
\item $A_B = \bigcup_{i<p}\big(\{i\}{\times}\sortS_i(B)\big)$ as a set.
\item For each $i<p$, $u^B_i:\sortS_i(B)\to A_B$ is given by $u_i^B(b) = (i,b)$.
\item $R_q^{A_B} = \left\{\big((i_0,b_0),...,(i_{r-1},b_{r-1})\big):\qtp^B(b_0,...,b_{r-1})=q\right\}$ for each irreflexive $q\in S_r^\qf(T_\KK)$, $r\leq\ari(\L)$.
\end{itemize}
\item We define $\tilde\KK$ to be the isomorphism-closure of $\left\{C: C\leq A_B, B\in\KK\right\}$.
\item For each $C\in\tilde\KK$, we define an $\L$-structure $B^C$ and a family of partial maps $v_i^C:C\pto \sortS_i(B^C)$ as follows:
\begin{itemize}
\item For $q\in S_1^\qf(T_\KK)$, let $i_q<p$ be the index such that $q\models\sortS_{i_q}$.
\item For each $i<p$, let $\sortS_i(B^C) = \bigcup\left\{R_q^C:q\in S_1^\qf(T_\KK),\,i_q = i\right\}$ and let $v_i^C$ be the identity mapping.
\item For $R\subseteq\sortS_{i_0}\times\cdots\times\sortS_{i_{r-1}}$ in $\sig(\L)$, let $Q_R =\left\{q\in S_r^\qf(T_\KK):q\models R(x_0,...,x_{r-1})\right\}$ and let $R^{B^C} = \bigcup\left\{R_q^C:q\in Q_R\right\}$.
\end{itemize}
\item For an $\L_\KK$-structure $\M$ such that $\Age(\M)\subseteq\tilde\KK$, we define $B^\M$ similarly.
\end{itemize}
\end{defn}

\begin{lemma}\label{Lemma_transfer_1sort}
Let $\KK$ be an algebraically trivial \fraisse class in a $p$-sorted language $\L$. Then:
\begin{enumerate}
\item If $B\in\KK$, then $B\cong B^{A_B}$ via $b\mapsto (i_b,b)$ where $i_\bullet:b\mapsto i_b$ is such that $b\in \sortS_{i_b}(B)$ for each $b$.
\item If $C\in\tilde\KK$, then $C \cong A_{B^C}$ via $c\mapsto (i_{\qtp(c)},c)$.
\end{enumerate}
\end{lemma}
\begin{proof}
The proofs of the two items of the lemma are very similar, so we will just prove item 1. Let $B\in\KK$ be given. The map $f: b\mapsto (i_b,b)$ is actually the union $f = \bigcup_{i<p}(v_i^{A_B}\circ u_i^B)$, and it is clear that $f$ is a bijection between $B$ and $B^{A_B}$. To see that $f$ is an isomorphism, let $R\subseteq\sortS_{i_0}\times\cdots\times\sortS_{i_{r-1}}$ be in $\sig(\L)$, and let $b_j\in \sortS_{i_j}(B)$ for each $j<r$. Let $q = \qtp^B(b_0,...,b_{r-1})$, so that $\bb\in R_q^{A_B}$ by definition. Then
\begin{align*}
B\models R(b_0,...,b_{r-1})\,&\iff\, R(x_0,...,x_{r-1})\in q\\
&\iff\, q\in Q_R\\
&\iff\,R_q^{A_B}\subseteq R^{B^{A_B}}\\
\end{align*}
and it follows that $B\models R(\bb)\iff B^{A_B}\models R(f\bb)$. This completes the proof.
\end{proof}

\begin{cor}
Let $\KK$ be an algebraically trivial \fraisse class in a $p$-sorted language $\L$. 
\begin{enumerate}
\item If $C\in\tilde\KK$, then $B^C\in \KK$.
\item $\tilde\KK$ is an algebraically trivial \fraisse class.
\item $\KK$ is the isomorphism-closure of $\left\{B^C:C\in\tilde\KK\right\}$
\end{enumerate}
\end{cor}
\begin{proof}
For Item 1: Given $C\in\tilde\KK$, by definition, there is some $B_0\in\KK$ such that $C\leq A_{B_0}$. One easily verifies that $B^C\leq B^{A_{B_0}}\cong B_0$, so as $\KK$ is a \fraisse class, we find that $B^C\in\KK$. 
For Item 2: HP for $\tilde\KK$ is built in to its definition, and for JEP and AP, one simply transfers the discussion from $\tilde\KK$ to $\KK$ via $C\mapsto B^C$, applies JEP or AP there, and transfers it back to $\tilde\KK$ via $B\mapsto A_B$.  Item 3 is immediate from Lemma \ref{Lemma_transfer_1sort}.
\end{proof}


The remainder of the proof of Theorem \ref{Thm_one_sort} is encoded in the following proposition, which extends to the transfer between the two \fraisse classes $\KK$ and $\tilde\KK$ to the level of their generic models.

\begin{prop}
Let $\KK$ be an algebraically trivial \fraisse class in a $p$-sorted language $\L$. Let $\B$ be the generic model of $\KK$, and let $\M$ be generic model of $\tilde\KK$. Then $B^\M\cong\M$. 

It follows that there are injections $u_i:\sortS_i(\B)\to M$ ($i<p$) and a surjective mapping $i_\bullet:S_1^\qf(T_{\tilde\KK})\to p:q\mapsto i_q$ such that:
\begin{itemize}
\item $M = \dot\bigcup_{i<p}\img(u_i)$, and $\img(u_i) = \bigcup\left\{q(\M):i_q=i\right\}$ for each $i<p$.
\item For each irreflexive $q\in S_r^\qf(T_\KK)$, $r\leq\ari(\L)$, if $q\models\sortS_{i_0}\times\cdots\times\sortS_{i_{r-1}}$ and $(b_0,...,b_{r-1})\in\prod_{j<r}\sortS_{i_j}(\B)$
$$\B\models q(b_0,...,b_{r-1})\iff \M\models R_q(u_{i_0}(b_0),...,u_{i_{r-1}}(b_{r-1})).$$
\end{itemize} 
\end{prop}
\begin{proof}
One verifies that $B^\M$ is $\KK$-universal and $\KK$-homogeneous, and that $A\in\KK$ whenever $A$ is a finite subset of $B^\M$. It follows that $B^\M\cong \B$ by the uniqueness of the generic model of an amalgamation class.
\end{proof}

\newpage
\section{Linear orderings and the connection to collapse-of-indiscernibles dividing-lines}\label{Sec_LO_Indisc}

The starting point for the research in this paper was an attempt to generalize the collapse-of-indiscernible dividing-lines results from \cite{guingona-hill-scow} and \cite{scow-2011}.  These results are themselves a generalization of Shelah's classification of stable theories in terms of indiscernible sequences: A theory is stable if and only if every indiscernible sequence is an indiscernible set \cite{shelah-ClassifTheory}.  In this section, we connect our discussion with this concept and explore the relationship between positive-local-combinatorial and collapse-of-indiscernible dividing-lines.  This begins with a quick overview of the Ramsey property, the Patterning property, and generalized indiscernibles from \cite{guingona-hill-scow}.  We then discuss order-expansions of \fraisse classes, showing that adding a generic order to a algebraically trivial \fraisse class with the order property does not change the corresponding dividing-line.  We conclude that, if $\KK$ is a Ramsey-expandable \fraisse class, then $\C_\KK$ is characterized by a non-collapse of $\KK^<$-indiscernibles.

\subsection{Definitions and previously known facts}

We begin our discussion by recalling definitions around generalized indiscernibles.  This starts with the definition of a Ramsey class; as it turns out these are exactly the classes which produce well-behaved indiscernibles (i.e., ones that have the Patterning property).  From there, we discuss (un-collapsed) indiscernible pictures.  We then define the Patterning property, and recall a theorem from \cite{guingona-hill-scow}, stating that the Pattering property is equivalent to the Ramsey property.

\begin{defn}[Ramsey property, Ramsey class]
Let $\KK$ be a \fraisse class.
\begin{itemize}
\item For $A\in\KK$, we say that $\KK$ has the \emph{$A$-Ramsey property} if, for any $0<k<\omega$ and any $B\in\KK$, there is some $C=C(A,B,k)\in\KK$ such that, for any coloring $\xi:\emb(A,C)\to k$, there is an embedding $u:B\to C$ such that $\xi$ is constant on $\emb(A,uB)$.

\item $\KK$ is said to have the \emph{Ramsey property} if it has the $A$-Ramsey property for every $A\in\KK$. When $\KK$ has the Ramsey property, then we also say that $\KK$ is a Ramsey class.
\end{itemize}
\end{defn}

\begin{defn}
Let $\KK$ be a \fraisse class with generic model $\A$. Let $\M$ be an infinite $\L$-structure for some language $\L$. 
\begin{itemize}
\item A {\em picture of $\A$ in $\M$}, $\gamma:\A\to\M$, is a just an injective mapping of $A$ into (a single sort of) $\M$.
\item A picture $\gamma:\A\to\M$ of $\A$ in $\M$ is {\em indiscernible} if for all $n\in\NN$, $a_0,...,a_{n-1}$ and $b_0,...,b_{n-1}$ in $A$,
$$\qtp^\A(\aa) = \qtp^\A(\bb)\,\,\implies\,\,\tp^\M(\gamma\aa) = \tp^\M(\gamma\bb).$$
(For $\Delta\subseteq\L$, $\Delta$-indiscernible pictures are defined similarly.)

\item An indiscernible picture $\gamma:\A\to\M$ is called {\em un-collapsed} if for all $n\in\NN$, $a_0,...,a_{n-1}$ and $b_0,...,b_{n-1}$ in $A$,
$$\tp^\M(\gamma\aa) = \tp^\M(\gamma\bb)\,\,\implies\,\,\qtp^\A(\aa) = \qtp^\A(\bb).$$
Of course, we say that $\gamma$ {\em collapses} if it is not un-collapsed.
\end{itemize}
Usually, we will denote indiscernible pictures with the letters $I$ or $J$ instead of $\gamma$. 
\end{defn}

\begin{defn}[Patterning property]
Let $\KK$ be a \fraisse class with generic model $\A$. Let $\M$ be an infinite $\L$-structure for some language $\L$.

Let $\gamma:\A\to\M$ be a picture, and let $I:\A\to\M$ be an indiscernible picture. We say that {\em $I$ is patterned on $\gamma$} if for every $\Delta\subset_\fin\L$, every $n\in\NN$, and all $a_0,...,a_{n-1}\in A$, there is an embedding $f = f_{\Delta,\aa}:\A\r\aa\to\A$ such that 
$$\tp_\Delta^\M(I\aa) = \tp_\Delta^\M(\gamma f\aa).$$
Now, we say that {\em $\KK$ has the Patterning property} if for every picture $\gamma:\A\to\M$, there is an indiscernible picture $I:\A\to\M$ of $\A$  patterned on $\gamma$.
\end{defn}

The existence of indiscernible sequences is usually stated (as in \cite{marker-textbook}) with less precision than is actually required in practice. The existence statement in full precision, but generalized to objects richer than pure linear orders, is the following theorem due to \cite{scow-2011}.

\begin{thm}\label{Thm_Lynn_patterning}
 $\KK$ has the Ramsey property if and only if it has the Patterning property.
\end{thm}

Thus, if we wish to consider \fraisse classes that produce a coherent theory of indiscernibles, we are compelled to look at Ramsey classes.  Furthermore, as the next theorem will show, looking at Ramsey classes forces us to consider algebraically trivial classes which carry a $0$-definable linear order.  Therefore, in this section, we will be primarily interested in studying algebraically trivial \fraisse classes that, when one adds a generic linear order, become Ramsey classes.

\begin{thm}
Let $\KK$ be a \fraisse class with disjoint-JEP, and let $\A$ be the generic model of $\KK$. If $\KK$ has the Ramsey property, then:
\begin{itemize}
\item (\cite{Nesetril03ramseyclasses}) $\KK$ is algebraically trivial.
\item (\cite{kpt-2005}) $\A$ carries a 0-definable linear ordering.
\end{itemize}
\end{thm}

\begin{defn}
Let $\KK$ be an algebraically trivial \fraisse class in the language $\L_\KK$, and assume that $|S_1(T_\KK)|=1$. Let $\L_\KK^<$ be the language obtained by adding one new binary relation $<$ to the signature of $\L_\KK$, and let $\KK^<$ be the class of all finite $\L_\KK^<$-structures $B$ such that $B\r\L_\KK\in\KK$ and $<^B$ is a linear ordering of $B$. Then $\KK^<$ is also an algebraically trivial \fraisse class with $|S_1(T_{\KK^<})|=1$, and if $\B$ is its generic model, then $<^\B$ is a dense linear ordering of $B$ without endpoints.

The class $\KK^<$ is the {\em generic order-expansion} of $\KK$. We will say that the original class $\KK$ is {\em simply Ramsey-expandable} if $\KK^<$ has the Ramsey property. (A more general notion of Ramsey-expandable would allow arbitrary expansions by finitely many relation symbols; this topic is addressed in, for example,
\cite{hubicka-nesetril-2014,hubiczka-neszetrzil-2016}, under the name of ``having a Ramsey-lift.'')
\end{defn}

\subsection{Order-Expansions}

\emph{A priori}, adding a generic linear order to an algebraically trivial \fraisse class could undermine the project of classifying dividing-lines by changing the corresponding class of theories.  Luckily, that seems not to be the case.  We will show, in Corollary \ref{Cor_unstable_yields_=}, that, given any algebraically trivial \fraisse class $\KK$, if $T_\KK$ has the order property, then $\KK$ and $\KK^<$ correspond to the same dividing-line.  Therefore, as long as the generic model of $\KK$ is unstable, we need not worry about adding a generic order to $\KK$.

We begin with the definition of ``coding orders,'' which is precisely what is needed to show that adding a generic order will have no effect on the corresponding dividing-line.  Then, we show that a class codes orders if and only if its generic model has the order property.

\begin{defn}
Let $\KK$ be an algebraically trivial \fraisse class of $\L$-structures. We say that {\em $\KK$ codes orders} if there are $0<n<\omega$, quantifier-free formulas $\theta_R(\xx_0,...,\xx_{r-1})$ ($R^{(r)}\in\sig(\L)$, $|\xx_i|=n$), and a quantifier-free formula $\theta_<(\xx,\yy)$ such that  for every $B\in\KK^<$, there are $C\in\KK$ and an injection $f:B\to C^n$ such that 
$$B\models R(b_0,...,b_{r-1})\iff C\models \theta_R(f(b_0),...,f(b_{r-1}))$$
for each $R^{(r)}\in\sig(\L)$ and all $b_0,...,b_{r-1}\in B$, and 
$$b<^Bb'\iff C\models\theta_<(f(b),f(b'))$$
for all $b,b'\in B$.
\end{defn}

\begin{obs}\label{Obs_codes_orders}
Let $\KK$ be an algebraically trivial \fraisse class that codes orders, and let ${\A^< = (\A,<)}$ be the generic model of $\KK^<$. Then there are $0<n<\omega$, quantifier-free formulas $\theta_R(\xx_0,...,\xx_{r-1})$ ($R^{(r)}\in\sig(\L)$, $|\xx_i|=n$), a quantifier-free formula $\theta_<(\xx,\yy)$, and an injection $u:A\to A^n$ such that 
$$\A\models R(a_0,...,a_{r-1})\iff \A\models \theta_R(u(a_0),...,u(a_{r-1}))$$
for each $R^{(r)}\in\sig(\L)$ and all $a_0,...,a_{r-1}\in A$, and 
$$a<a '\iff \A\models\theta_<(u(b),u(b'))$$
for all $a,a'\in A$.
\end{obs}

\begin{prop}
 Let $\KK$ be an algebraically trivial \fraisse class.  Then, $\KK$ codes orders if and only if $T_\KK$ has the order property.
\end{prop}
\begin{proof}
 Obviously, if $\KK$ codes orders,  $\theta_<(\xx,\yy)$ as in the definition of coding orders has the order property in $T_\KK$.

Coversely, suppose $T_\KK$ has the order property, and let $\A \models T_\KK$ be the generic model.  Since $\A$ has the order property and $T_\KK$ eliminates quantifiers, there is a quantifier-free $\L$-formula $\psi(\xx;\yy)$, $|\xx|=|\yy|=m$, such that for every $n$, there is a sequence $( \aa_i )_{ i < n}$ of $m$-tuples from $\A$ such that, for all $i,j < n$,
 $$  \A \models \psi(\aa_i; \aa_j) \iff i < j.$$
Let $\theta_<(x'\xx,y'\yy) = \psi(\xx,\yy)$, and for each $R^{(r)}\in\sig(\L)$, let 
$$\theta_R(x'_0\xx_0,x'_1\xx_1,...,x_{r-1}'\xx_{r-1}) = R(x'_0,...,x'_{r-1})$$

Now, let $B \in \KK^<$ be given -- and say, $n = |B|$ and  $B = \{ b_0<^B\cdots<^Bb_{n-1} \}$.   Since $\A$ is the generic model of $\KK$, we can choose a sequence $( \aa_i )_{ i < n}$ of $m$-tuples from $\A$ such that $\A \models \psi(\aa_i; \aa_j) \iff i < j$ for all $i,j < n$, and an embedding $u_0:B\r\L\to\A$. We define $f:B\to A^{m+1}$ by setting $f(b_i)=\left(u_0(b_i),\aa_i\right)$ for each $i<n$, and we take $C$ to be the induced substructure of $\A$ on $u_0B\cup\bigcup_i\aa_i$. It is routine to verify that $C$ and $f$ meet the requirements of coding orders for $B\in\KK^<$. As $B$ was arbitrary, $\KK$ indeed codes orders.
\end{proof}

These are the ingredients needed to prove the main result of this subsection: For an algebraically trivial \fraisse class whose generic model is unstable, adding a generic linear order does not change the corresponding dividing-line.

\begin{cor}\label{Cor_unstable_yields_=}
Let $\KK$ be an algebraically trivial \fraisse class with generic model $\A$ and generic order-expansion $\KK^<$ (whose generic model is $\A^< = (\A,<)$). If $T_\KK$ is unstable, then 
$\C_\KK=\C_{\KK^<}$.
\end{cor}
\begin{proof}
$\C_{\KK^<} \subseteq \C_\KK$ is trivial because $\A^<$ is an expansion of $\A$. 
To show $\C_\KK \subseteq \C_{\KK^<}$, let $T\in\C_\KK$, $\M\models T$ saturated, and let $\pphi = \left(\phi_0,...,\phi_{m-1}\right)$ and $F:\A\to\M$ be a $\pphi$-resolved member of $\vec\FF_\pphi(T)$ such that $\A_\pphi(F)=\A$ up to an identification of relation symbols. Also, let $0<n<\omega$, $\theta_R$, $\theta_<$, and $u:A\to A^n$ be as described in Observation \ref{Obs_codes_orders}. It is not hard to see that $T_{\KK^<}\in Q_T$ via 
$$F' = (\underset{\textnormal{$n$ times}}{\underbrace{F,...,F}}){\circ}u$$
so $T\in\C_{\KK^<}$. 
\end{proof}

We can apply this result to the \fraisse class of finite sets with $k$ independent linear orders.  These classes were studied, for example, in Section 3 of \cite{guingona-hill-scow} in the classification of op-dimension.  It turns out that the dividing-line corresponding to $k$ independent linear orders is the same as the dividing line corresponding to one linear order.

\begin{defn}
For $0<k<\omega$, let $\L_k$ be the language whose signature consists of binary relation symbols $<_0,...,<_{k-1}$. Let $\MO_k$ be the class of all finite $\L_k$-structures $B$ such that $<_i^B$ is a linear ordering of $B$ for each $i<k$. Members of $\MO_k$ are called $k$-multi-orders, and it is not difficult to verify that $\MO_k$ is an algebraically trivial \fraisse class. Clearly, $\MO_1$ is just the class $\mathbf{LO}$ of all finite linear orders.
\end{defn}

\begin{prop}\label{Prop_MOandLO}
For every $0<k<\omega$, $\MO_{k+1}$ is the generic order-expansion of $\MO_k$: $\MO_{k+1}=\MO_k^<$. Since each $\MO_k$ is unstable, we find that $\C_{\MO_{k+1}}=\C_{\MO_k^<}=\C_{\MO_k}$. Thus, $\C_{\MO_k}=\C_\mathbf{LO}$ for every $k>0$.
\end{prop}

This is, perhaps, not surprising.  As noted in Section 3 of \cite{guingona-hill-scow}, $\MO_k$ relates to op-dimension $\ge k$, and a theory $T$ has sorts of arbitrarily high op-dimension if and only if $T$ is unstable (i.e., $T \in \C_\mathbf{LO}$).

\begin{obs}
$\mathbf{LO}$ is indecomposable. 
\end{obs}

\subsection{Collapse of indiscernibles}

In this subsection, we make explicit the connection between positive-local-combinatorial and collapse-of-indiscernible dividing-lines.  In Theorem \ref{Thm_collapse_char}, we show that the dividing-line corresponding to an algebraically trivial \fraisse class with the Ramsey property is characterized by the existence of an un-collapsed indiscernible picture.  In Theorem \ref{Thm_collapse_char_2}, we generalize this to simply Ramsey-expandable classes.  This partially captures a generalization of the standard collapse-of-indiscernible results found in the literature \cite{chernikov-et-al, guingona-hill-scow,scow-2011}.

\begin{thm}\label{Thm_collapse_char}
Let $\KK$ be an algebraically trivial \fraisse class with generic model $\A$ in a language $\L$. If $\KK$ is a Ramsey class, then for every $T\in\TT$, the following are equivalent:
\begin{enumerate}
\item $T\in\C_\KK$
\item There is an un-collapsed indiscernible picture of $\A$ in a model of $T$.
\end{enumerate}
\end{thm}
\begin{proof}
2$\implies$1 is trivial. For 1$\implies$2, let $T\in\TT$, $\M\models T$ saturated, and let $\pphi = \left(\phi_0,...,\phi_{m-1}\right)$ and $F:\A\to\M$ be a $\pphi$-resolved member of $\vec\FF_\pphi(T)$ such that $\A_\pphi(F)=\A$ up to an identification of relation symbols.

Let $I:\A\to\M$ be an indiscernible picture of $\A$ patterned on $F$. We claim that $I$ is un-collapsed. Towards a contradiction, suppose there are $0<k<\omega$ and $\aa,\aa'\in A^k$ such that $\qtp^\A(\aa)\neq\qtp^\A(\aa')$ but $\tp^\M(\aa)=\tp^\M(\aa')$. Since $\pphi$ is finite and $I$ is patterned on $F$, there is an $\L$-embedding $u:\aa\aa'\to\A$ such that $\tp_\pphi^\M(I\aa\aa')=\tp_\pphi^\M(Fu\aa\aa')$, and it follows that 
$$\tp_\pphi^\M(Fu\aa)=\tp_\pphi^\M(I\aa)=\tp_\pphi^\M(I\aa')=\tp_\pphi^\M(Fu\aa').$$
Since $\A = \A_\pphi(F)$, we find that $\qtp^\A(u\aa)=\qtp^\A(u\aa')$, and since $u$ is an embedding and $\A=\A_\pphi(F)$, we find that $\qtp^\A(\aa)=\qtp^\A(\aa')$ -- a contradiction. Thus, $I$ must be un-collapsed.
\end{proof}

Though Theorem \ref{Thm_collapse_char} is certainly an interesting result in its own right, it does not match the flavor of collapse-of-indiscernible results discovered thus far (in, say, \cite{chernikov-et-al}, \cite{guingona-hill-scow}, and \cite{scow-2011}).  ``Natural'' dividing-lines typically do not involve linear orders.  For example, the independence property naturally corresponds to the \fraisse class of finite graphs (and not finite ordered graphs).  However, in light of Corollary \ref{Cor_unstable_yields_=}, adding generic linear orders should not matter.  Therefore, we get a stronger result, more in the spirit of typical collapse-of-indiscernibles results.

\begin{thm}\label{Thm_collapse_char_2}
Let $\KK$ be an algebraically trivial \fraisse class. If $T_\KK$ is unstable and $\KK$ is simply Ramsey-expandable, then the following are equivalent:
\begin{enumerate}
\item $T\in\C_\KK$
\item There is an un-collapsed indiscernible picture of $\A^<$ in a model of $T$, where $\A^<$ is the generic model of $\KK^<$.
\end{enumerate}
\end{thm}
\begin{proof}
By Theorem \ref{Thm_collapse_char} and Corollary \ref{Cor_unstable_yields_=}.
\end{proof} 

Theorem \ref{Thm_collapse_char_2} is, in spirit, a generalization of the collapse-of-indiscernible results from \cite{chernikov-et-al}, \cite{guingona-hill-scow}, and \cite{scow-2011}.  The following example explores this connection.


\begin{example}
 By choosing the appropriate $\KK$, we obtain the following corollaries of Theorem \ref{Thm_collapse_char_2}:
 \begin{enumerate}
  \item If $\KK = \HH_{r+1}$ is the class of finite $(r+1)$-hypergraphs, then $\C_\KK$ is the class of theories with $r$-IP (see Proposition \ref{Prop_Chernikov}).  Therefore, $T$ has $r$-IP if and only if there is an un-collapsed indiscernible picture of a model of the generic ordered $(r+1)$-hypergraph in a model of $T$.  This is analogous to Theorem 5.4 in \cite{chernikov-et-al}.
  \item In particular, consider $r = 1$.  Then, $T$ has IP if and only if there is an un-collapsed indiscernible picture of a model of the generic ordered graph in a model of $T$.  This is analogous to the main result of \cite{scow-2011}.
 \end{enumerate}
\end{example}

One thing missing from the discussion here is precisely how these indiscernibles resist collapse.  The results from \cite{chernikov-et-al}, \cite{guingona-hill-scow}, and \cite{scow-2011} all provide a more precise reason for the lack of collapse.  For example, in \cite{scow-2011}, we see that a theory has NIP if and only if all ordered graph indiscernibles collapse to ordered indiscernibles.  It is precisely the collapse of the edge relation that determines whether or not a theory has NIP.  In future work, it would be interesting to explore exactly what reducts of $\A^<$ characterize whether or not a theory belongs to $\C_\KK$ (in Theorem \ref{Thm_collapse_char_2}).

\newpage
\section{Case studies}\label{Sec_Case_studies}

In this section, to get a better feel for the quasi-ordering $\cleq$, we study some examples of classes in the $\cleq$-ordering.  The first one we examine is the $\cleq$-maximal class, which includes, for example, the theory of hereditarily finite sets.  Next, we explore hypergraphs.  We show that prohibiting cliques of a certain size does not change the corresponding dividing-line, nor does adding a generic linear order.  We show that hypergraphs form a strict infinite ascending $\cleq$-chain, and use this to show that there are infinitely many irreducible dividing-lines.  In the next subsection, we look at societies, showing that adding generic hyperedge relations of lower or equal arity to hypergraphs does not change the corresponding dividing-line.  Finally, we look at how multi-partite multi-concept classes relate to hypergraphs, showing that the dividing-lines corresponding to hypergraphs are precisely the higher-order independence properties.

\subsection{The top of the $\cleq$-order}

In some sense, the top of the $\cleq$-order is not surprising.  If this is a sensible ordering on the local complexity of theories, the maximal class must include the theory of true arithmetic.  Moreover, the top class must also include the theory of hereditarily finite sets since these two theories are bi-interpretable.  We show that every algebraically trivial \fraisse class is subordinate to this theory, and we conclude from this that all theories lie $\cleq$-below it.  Finally, we give an example of a theory in the top class which interprets neither hereditarily finite sets nor true arithmetic.

\begin{defn}
Let $\L_0$ be the language of set theory, with signature $\{\in\}$. We define $\H = (H,\in^\H)$, the hereditarily finite sets, as follows:
\begin{itemize}
\item $H_0 = \emptyset$, and for each $n<\omega$, $H_{n+1} = H_n\cup\P(H_n)$
\item Then $H = \bigcup_nH_n$, and $\in^\H$ is the usual membership relation.
\end{itemize}
We write $T_\fs$ for the complete theory of $\H$ -- the theory of finite sets.
\end{defn}

\begin{prop}\label{Prop_top}
$T_\KK\cleq T_\fs$ for every algebraically trivial \fraisse class $\KK$, and the sequence of formulas involved depends only on the signature of $\KK$, not on $\KK$ itself. Thus, $T\cleq T_\fs$ for every $T\in\TT$.

\end{prop}
\begin{proof}
Let $\KK$ be an algebraically trivial \fraisse class in a language $\L$ with signature $\sig(\L) = \left\{R_0,...,R_{n-1}\right\}$, and let $\A$ be its generic model. Let $m = n+1$, and for each $i<n$, if $r = \ari(R_i)$, then let $\theta_i(\xx_0,...,\xx_{r-1})$, where $|\xx_i|=m$, be the formula asserting
$$(x_{0,0},...,x_{r-1,0})\in x_{0,i+1}.$$
Now, if $B\in\KK$, we define a certain injection $u:B\to H$. First, let $u_0:B\to H$ be any injection at all, and then define $u:B\to H$ by 
$$u(b) = \left(u_0(b),u_0{\left[R_0^B\right]},...,u_0{\left[R^B_{n-1}\right]}\right).$$
We observe that for each $i<n$, for all $b_0,...,b_{r-1}\in A$ where $r = \ari(R_i)$, 
\begin{align*}
B\models R_i(b_0,...,b_{r-1}) &\iff \left(u_0(b_0),...,u_0(b_{r-1})\right)\in u_0{\left[R_i^B\right]}\\
&\iff \left(u(b_0)_0,...,u(b_{r-1})_0\right)\in u(b_0)_{i+1}\\
&\iff \H\models\theta_i(u(b_0),...,u(b_{r-1})).
\end{align*}
Since $B$ was an arbitrary member of $\KK$, by compactness, there are an elementary extension $\H'$ of $\H$ and an injection $F:A\to H'$ such that for all $i<n$, $r = \ari(R_i)$, for all $a_0,...,a_{r-1}\in A$, 
$$\A\models R_i(a_0,...,a_{r-1})\iff \H'\models\theta_i(F(a_0),...,F(a_{r-1})).$$
It follows that $T_\KK\cleq T_\fs$.
\end{proof}

Using the bi-interpretablility of true arithmetic with the theory of hereditarily finite sets, we conclude that true arithmetic also lives in the top class.

\begin{cor}
Let $\mathit{TA} = \Th(\NN,+,\cdot,0,1)$, often called ``true arithmetic.'' Since $T_\fs$ is interpretable in $\mathit{TA}$, it follows that $T\cleq \mathit{TA}$ for every $T\in\TT$.
\end{cor}
\begin{proof}
In general, we know that for any $T_1,T_2\in\TT$, if $T_1$ is interpretable in $T_2$, then $T_1\cleq T_2$. The corollary then follows immediately from Theorem \ref{Prop_top} and the fact that $T_\fs$ and $\mathit{TA}$ are bi-interpretable. 
\end{proof}

Examining the proof of Corollary \ref{Cor_continuum_bound} and the proof of Proposition \ref{Prop_indec_yields_irred} gives another example of a theory $T^*$ that belongs to the $\cleq$-maximal class.  However, $T^*$ does not interpret $T_\fs$ or $\mathit{TA}$, showing that there are theories in the top class that are not bi-interpretable.

\begin{defn}
Following the proof of Corollary \ref{Cor_continuum_bound}, let $S$ be the set of functions $s:\omega\to\omega$ of finite support. For each $s\in S$, let $\sig(\L_s)$ be the signature with relation symbols $R_i^{(s(i))}$ for each $i\in \mathrm{supp}(s)$, and let $\K_s$ be the set of all algebraically trivial \fraisse classes of finite $\L_s$-structures. Then, let $T^*$ be the theory $T_S$ constructed from $S = \left\{T_\KK:\KK\in\K_s, s\in S\right\}$ as in the proof of Proposition \ref{Prop_indec_yields_irred}. 
\end{defn}

\begin{fact}
$T\cleq T^*$ for every $T\in\TT$, so in particular, $T_\fs,\mathit{TA}\cleq T^*$. However $T^*$ does not interpret $T_\fs$ or $\mathit{TA}$.
\end{fact}
\begin{proof}
For the first part of the claim, we notice that for every algebraically trivial \fraisse class $\KK$, $T_\KK\cleq T^*$ by construction. For non-interpretability, just notice that any reduct of $T^*$ to finitely many of its sorts is $\aleph_0$-categorical, but obviously $T_\fs$ is not $\aleph_0$-categorical.
\end{proof}

\subsection{Hypergraphs}

In this subsection, we explore hypergraphs, showing that the \fraisse class of $r$-hypergraphs is indecomposible (i.e., corresponds to an irreducible dividing-line).  We show that adding a generic order or prohibiting cliques of a fixed size does not alter the corresponding dividing-line.  We also show that these hypergraph classes form a strict chain.  We begin the discussion with the definition of the relevant algebraically trivial \fraisse classes, $\HH_r$, $\HH_{r,k}$, and $\HH^*_r$.

\begin{defn}
Let $2\leq r<\omega$, and let $\L_r$ be the language whose signature is a single $r$-ary relation symbol $R$. 
\begin{itemize}
\item $\HH_r$ is the class of all finite $r$-hypergraphs -- i.e., finite models of the two sentences
$$\forall \xx\left(R(\xx)\cond\bigwedge_{i<j}x_i\neq x_j\right),\,\,\forall \xx\left(R(\xx)\cond\bigwedge_{\sigma\in\sym(r)}R(x_{\sigma(0)},...,x_{\sigma(r-1)})\right).$$
\item For $k>r$, $\HH_{r,k}$ is the sub-class consisting of those $A\in\HH_r$ that exclude $K_{k}(r)$, the complete $r$-hypergraph on $k$ vertices. So, $\HH_{r,k}$ is the class of finite models of the previous two sentences and the sentence
$$\forall x_0,...,x_{k-1}\left(\bigwedge_{i<j}x_i\neq x_j\cond\bigvee_{i_0<\cdots<i_{r-1}<k}\neg R(x_{i_0},...,x_{i_{r-1}})\right)$$
\end{itemize}
One easily verifies that all $\HH_r$'s and $\HH_{r,k}$'s are algebraically trivial \fraisse classes. 

Now, let $\L_r^*$ be the expansion of $\L_r$ with new unary predicate symbols $U_0,...,U_{r-1}$, and let $\HH^*_r$ be the set of models of the first two sentences and the sentences
$$\forall x\bigvee_{i<r}\left(U_i(x)\wedge \bigwedge_{j\neq i}\neg U_j(x)\right),$$
$$\forall x_0...x_{r-1}\left(R(\xx)\cond \bigvee_{\sigma\in\sym(r)} \bigwedge_{i < r} U_{\sigma(i)}(x_i)\right)$$
Again, it is not hard to see that $\HH^*_r$ is, again, an algebraically trivial \fraisse class.
\end{defn}

We make a few observations about the relationship between these \fraisse classes.

\begin{obs}
Let $2\leq r<k<\omega$, and let $\A^<,\B^<$ be the generic models of $\HH_{r,k}^<$, $\HH_r^<$, respectively, and $\A = \A^<\r\L_r$, $\B=\B^<\r\L_r$. Then, there are embeddings $\A\to\B$ and $\A^<\to\B^<$. It follows that $\C_{\HH_r}\subseteq\C_{\HH_{r,k}}$ and $\C_{\HH_r^<}\subseteq\C_{\HH_{r,k}^<}$ whenever $2\leq r<k<\omega$.
\end{obs}

\begin{obs}
For every $2\leq r<\omega$, $\C_{\HH_{r}^*} =\C_{\HH_{r}}$. (Because every $T\in\TT$ eliminates imaginaries.)
\end{obs}

As promised, we show that the \fraisse class of hypergraphs of a fixed arity is indecomposable.  Since this is an algebraically trivial \fraisse class, it corresponds to an irreducible dividing-line (by Theorem \ref{Thm_prime_equals_fraisse}).

\begin{prop}
For every $2\leq r<\omega$, $\HH_r$ is indecomposable.
\end{prop}
\begin{proof}
Let $(\KK_0,...,\KK_{n-1})$ be a factorization of $\HH_r$. Let $\A$ be the generic model of $\HH_r$, and for each $i<n$, let $\A_i$ be the generic model of $\KK_i$. Let $\B$ be the generic model of $\Pi_i\KK_i$, and let $u:A\to B$ be an injection such that for all $k$ and all $\aa,\aa'\in A^k$, $\qtp^\A(\aa)=\qtp^\A(\aa')\iff\tp^\B(u\aa) = \tp^\B(u\aa')$. We make three observations:
\begin{itemize}
\item If $k<r$ and $\aa,\aa'\in A^k$ are both non-repeating, then $\qtp^\A(\aa) = \qtp^\A(\aa')$.
\item If $\aa\in A^r$ is non-repeating, then $\qtp^\A(\aa) = \qtp^\A(a_{\sigma(0)},...,a_{\sigma(r-1)})$ for any $\sigma\in\sym(r)$.
\item For any $k$ and any $\aa\in A^k$, $\qtp^\A(\aa) = \bigcup_{I\in{k\choose \leq r}}\qtp^\A(\aa_{\r I})$.
\end{itemize}
Let $\aa\in R^\A$ and $\aa'\in A^r\setminus R^\A$ non-repeating, and for each $i<n$, let $\theta_i(x_0,...,x_{r-1}) = \tp^\B(u\aa)\r\L_i$. Then for some $i_0<n$, $\neg\theta_{i_0}\in\tp^\B(u\aa')$. From this, we see that $\theta_{i_0}$ witnesses $T_{\HH_r}\cleq T_{\KK_{i_0}}$.
\end{proof}

We now establish the fact that hypergraphs form a strict chain (Corollary \ref{Cor_hyp_desc_chain}).  In Lemma \ref{Lemma_contain_hypergraphs}, we show that classes defined from hypergraphs form a decreasing chain and, in Lemma \ref{Lemma_strict_hypergraphs}, we establish that this chain is strict.  Together, these yield Theorem \ref{Thm_hypergraph_containment} and Corollary \ref{Cor_hyp_desc_chain}.  In particular, this shows that there are infinitely many irreducible dividing-lines, as noted in Corollary \ref{Cor_InfiniteClasses}.

\begin{lemma}\label{Lemma_contain_hypergraphs}
For all $2\leq r_1<r_2<\omega$, $\C_{\HH_{r_2}}\subseteq\C_{\HH_{r_1}}$. 
\end{lemma}

\begin{proof}
 For $i=1,2$, let $\A_i$ be the generic model of $\HH_{r_i}$.  Consider the $\L_{r_2}$-formula
 \[
  \theta(x_0 \yy_0, ..., x_{r_1-1} \yy_{r_1-1}) = R(x_0, ..., x_{r_1-1}, \yy_0)
 \]
 where $|\yy_t| = r_2 - r_1$ for all $t < r_1-1$.  For any $B_1 \in \HH_{r_1}$, construct an $\L_{r_2}$-structure $B_2$ as follows:
 \begin{itemize}
  \item As a set, $B_2 = B_1 \cup \{ \cc \}$ for some $\cc$ where $|\cc| = r_2 - r_1$ and $\cc \cap B_1 = \emptyset$.
  \item $R^{B_2}$ is $\{ (b_0, ..., b_{r_1-1}, \cc) : (b_0, ..., b_{r_1-1}) \in R^{B_1} \}$, closed under symmetry.
  \item Consider the injection $u : B_1 \rightarrow B_2^{r_2 - r_1 + 1}$ given by $u(b) = (b, \cc)$.
 \end{itemize}
 Then, for all $b_0, ..., b_{r_1-1} \in B_1$,
 \[
  B_1 \models R(b_0, ..., b_{r_1-1}) \iff B_2 \models \theta(u(b_0), ..., u(b_{r_1-1})).
 \]
 By compactness, we get an injection $u : A_1 \rightarrow A_2^{r_2 - r_1 + 1}$ such that, for all $a_0, ..., a_{r-1} \in A_1$
 \[
  A_1 \models R(a_0, ..., a_{r_1-1}) \iff A_2 \models \theta(u(a_0), ..., u(a_{r_1-1})).
 \]
 Therefore, $T_{\HH_{r_1}} \cleq T_{\HH_{r_2}}$, so $\C_{\HH_{r_2}} \subseteq \C_{\HH_{r_1}}$.
\end{proof}

\begin{lemma}\label{Lemma_strict_hypergraphs}
If $2\leq r_1<r_2<\omega$, then $\C_{\HH_{r_1}}\nsubseteq\C_{\HH_{r_2}}$.
\end{lemma}
\begin{proof}
For $i=1,2$, let $\A_i$ be the generic model of $\HH_{r_i}$. Towards a contradiction, suppose $\C_{\HH_{r_1}}\subseteq \C_{\HH_{r_2}}$, (i.e., $T_{\HH_{r_2}} \cleq T_{\HH_{r_1}}$) -- that is, there are $0<m<\omega$, an injection $u:A_2\to A_1^m$, and a quantifier-free formula $\theta(\xx_0,...,\xx_{r_2-1})$ such that $|\xx_i| = m$ for each $i<r_2$, and for all $a_0,...,a_{r_2-1}$ in $A_2$, 
$$\A_2\models R(a_0,...,a_{r_2-1})\iff \A_1\models\theta(u(a_0),...,u(a_{r_2-1}))$$
where $R$ is the single $r_2$-ary relation symbol of the language $\HH_{r_2}$. Since $r_1<r_2$, we may choose a number $N<\omega$ such that ${Nm\choose r_1}<{N\choose r_2}$ -- so that $2^{Nm\choose r_1}<2^{N\choose r_2}$. Now, let $B_0,...,B_{k-1}$ (where $k = 2^{N\choose r_2}$) be an enumeration of $N$-element substructures of $\A_2$ up to isomorphism -- that is, an enumeration of all $r_2$-hypergraphs on $N$ vertices. For each $i<k$, let $B_i^u$ be the $r_2$-hypergraph with universe $u[B_i]\subset A_1^m$ and interpretation $R^{B_i^u} = \theta(\A_1)\cap u[B_i]^m$. Then we find that 
$$2^{N\choose r_2}\leq \big|\left\{B_i^u:i<k\right\}/\cong\big|\leq 2^{Nm\choose r_1}<2^{N\choose r_2}$$
which is impossible. Thus $\C_{\HH_{r_1}}\nsubseteq \C_{\HH_{r_2}}$, as claimed.
\end{proof}

\begin{thm}\label{Thm_hypergraph_containment}
For all $2\leq r_1<r_2<\omega$, $\C_{\HH_{r_2}}\subsetneq\C_{\HH_{r_1}}$. 
\end{thm}
\begin{proof}
By Lemmas \ref{Lemma_contain_hypergraphs} and \ref{Lemma_strict_hypergraphs}.
\end{proof}

\begin{cor}\label{Cor_hyp_desc_chain}
$\C_{\HH_2}\supsetneq\C_{\HH_3}\supsetneq\cdots\supsetneq\C_{\HH_r}\supsetneq\cdots$. Thus, there is a strict nested chain of irreducible dividing-lines.
\end{cor}

\begin{cor}\label{Cor_InfiniteClasses}
If $\mathscr{F}$ denotes the family of all irreducible dividing-lines, then $\aleph_0\leq \big|\mathscr{F}\big|\leq 2^{\aleph_0}$.
\end{cor}
\begin{proof}
Combine Corollaries \ref{Cor_continuum_bound} and \ref{Cor_hyp_desc_chain}
\end{proof}

Now that we have established that the $\C_{\HH_r}$'s form a strictly decreasing chain, one might wonder what the intersection of these classes looks like.  It turns out that there is a theory that characterizes the intersection of all classes corresponding to hypergraphs, $\bigcap_{2\leq r<\omega}\C_{\HH_r}$.  We build this theory in the obvious manner, by disjointly ``gluing'' together the generic hypergraphs of each arity.

\begin{defn}
Let $\L$ be the language with,  for each $2\leq r<\omega$, a sort $\sortH_r$ and an $r$-ary relation $R_r$ on $\sortH_r$. Let $\M$ be the $\L$-structure such that $\sortH_r(\M)$ is the generic model of $\HH_r$ for each $2\leq r<\omega$, and let $T_\textsc{hyp} = \Th(\M)$.
\end{defn}

\begin{obs}
Let $T\in\TT$. The following are equivalent:
\begin{enumerate}
\item $T_\textsc{hyp}\cleq T$;
\item $T\in\bigcap_{2\leq r<\omega}\C_{\HH_r}$.
\end{enumerate}
\end{obs}

We return our attention to exploring the relationship between $\HH_r$, $\HH_{r,k}$, $\HH_r^<$, and $\HH_{r,k}^<$, culminating in Theorem \ref{Thm_exc_is_nothing}, which states that they all correspond to the same irreducible dividing line.

\begin{obs}\label{Obs_HHrHHrkOP}
For every $2\leq r<k<\omega$, $\HH_{r}$ and $\HH_{r,k}$ code orders  because $T_{\HH_r}$ and $T_{\HH_{r,k}}$ both have the order property. Hence, $\C_{\HH_r} = \C_{\HH_r^<}$ and $\C_{\HH_{r,k}} = \C_{\HH_{r,k}^<}$.
\end{obs}

Next, we show that prohibiting a $k$-clique has no effect on the corresponding dividing line.

\begin{lemma}
Let $2\leq r<k<\omega$. Then there is a quantifier-free formula $\theta(\xx_0,...,\xx_{r-1})$, $|\xx_i|=k-1$ such that for every $B\in\HH_r$, there are $C\in\HH_{r,k}$ and an injection $u:B\to C^{k-1}$ such that for all $b_0,...,b_{r-1}\in B$, 
$$B\models R(b_0,...,b_{r-1})\iff C\models\theta(u(b_0),...,u(b_{r-1})).$$
\end{lemma}
\begin{proof}
Let $\theta(\xx_0,...,\xx_{r-1}) = $
$$\bigwedge_{\textnormal{inj. $s:r\to k{-}1$}}R(x_{s(0),0},...,x_{s(r-1),r-1}).$$
Given $B\in\HH_r$, we define $C\in\HH_{r}$ as follows:
\begin{itemize}
\item $C = B\times\{0,1,...,k-2\}$ as a set.
\item $R^C = $
$$\left\{\big((b_0,i_0),...,(b_{r-1},i_{r-1})\big):\begin{array}{l}
(b_0,...,b_{r-1})\in R^B,\\
i_0,...,i_{r-1}<k-1\textnormal{ pairwise distinct}
\end{array}\right\}$$
\end{itemize}
We claim that $C$ is $K_k(r)$-free, where $K_k(r)$ is the complete $r$-hypergraph on $\{0,1,...,k-1\}$. Towards a contradiction, suppose $X\in{C\choose k}$ is such that $C\r X\cong K_k(r)$. Let $(b_0,i_0),...,(b_{k-1},i_{k-1})$ be pairwise distinct elements of $X$. By the pigeonhole principle, there are $s<t<k$ such that $i_s=i_t$. Selecting pairwise distinct $j_0,...,j_{r-3}\in k\setminus\{s,t\}$ arbitrarily, we have
$$\big((b_{j_0},i_{j_0}),...,(b_{j_{r-3}}, i_{j_{r-3}}),(b_s,i_s),(b_t,i_t)\big)\in R^C$$
because $C\r X$ is complete. But this contradicts the definition of $R^C$. Thus, $C$ is $K_k(r)$-free as claimed.

Finally, we define $u:B\to C^{k-1}$ by setting $u(b) = \big((b,0),...,(b,k-2)\big)$. For $b_0,...,b_{r-1}\in B$, we see that for every injection $s:r\to k{-}1$
$$(b_0,...,b_{r-1})\in R^B \iff \big((b_0,s(0)),(b_1,s(1)),...,(b_{r-1},s(r-1))\big)\in R^C$$
so 
$$(b_0,...,b_{r-1})\in R^B \iff C\models\theta(u(b_0),...,u(b_{r-1}))$$
as desired.
\end{proof}

\begin{cor}
Let $2\leq r<k<\omega$. Then there are a quantifier-free formula $\theta(\xx_0,...,\xx_{r-1})$ ($|\xx_i|=k-1$) and
an injection $u:A_r\to A_{r,k}^{k-1}$ such that 
for all $b_0,...,b_{r-1}\in B$, 
$$\A_r\models R(b_0,...,b_{r-1})\iff \A_{r,k}\models\theta(u(b_0),...,u(b_{r-1})).$$
It follows that $\C_{\HH_{r,k}}\subseteq \C_{\HH_r}$.
\end{cor}

By combining these ideas, we get the desired theorem, showing that the irreducible dividing-line corresponding to the class of all finite $r$-hypergraphs is also characterized by any one of the other classes of (ordered) $r$-hypergraphs mentioned above.

\begin{thm}\label{Thm_exc_is_nothing}
$\C_{\HH_r} = \C_{\HH_r^<} = \C_{\HH_{r,k}}=\C_{\HH_{r,k}^<}$ whenever $2\leq r<k<\omega$.
\end{thm}

\begin{proof}
Clearly $\HH_{r,k} \subseteq \HH_r$, so $\C_{\HH_r} \subseteq \C_{\HH_{r,k}}$.  By the previous corollary, we conclude $\C_{\HH_r} = \C_{\HH_{r,k}}$.  The remainder of the theorem follows by Observation \ref{Obs_HHrHHrkOP}.
\end{proof}

\begin{rem}
From Theorem \ref{Thm_exc_is_nothing}, we learn that the class/dividing-line $\C = \{\textnormal{unsimple theories}\}$ is not irreducible in the sense of this paper. To see this, suppose $\C$ {\em were} irreducible -- say $\C = \C_\KK$ for some indecomposable algebraically trivial \fraisse class $\KK$. Since $T_{\HH_{2,3}}$, the theory of the Henson graph, is in $\C$, we find that $T_\KK\cleq T_{\HH_{2,3}}\cleq T_{\HH_2}$, so $T_{\HH_2}\in\C$ -- i.e., $T_{\HH_2}$ is unsimple. But $T_{\HH_2}$ is the theory of the random graph, which certainly {\em is} simple -- a contradiction.
\end{rem}

\subsection{Societies}

Theorem \ref{Thm_exc_is_nothing} tells us that, given any $k > r \ge 2$, the class of theories corresponding to the algebraically trivial \fraisse class of finite (ordered) $r$-hypergraphs (omitting $k$-cliques) coincide.  What happens if we have more than one hyperedge relation, each acting independently?  Do we get a new dividing-line?  As it turns out, we get nothing new; adding new hyperedge relations of smaller or equal arity does not change the corresponding dividing-line.  We begin by formally defining the notion of a society, which captures the idea of having multiple independent hyperedge relations.

\begin{defn}
Let $\L$ be a finite relational language in which all relation symbols have arity $\geq 2$. For each $R^{(n)}\in\sig(\L)$, let $\phi_R$ be the sentence
$$\forall \xx\left(R(\xx)\cond\bigwedge_{i<j<n}x_i\neq x_j\right)\wedge\forall\xx\left(R(\xx)\cond\bigwedge_{\sigma\in\sym(n)}R(x_{\sigma(0)},...,x_{\sigma(n-1)})\right)$$
and let $\Sigma_\L$ for the set of sentences $\left\{\phi_R:R\in\sig(\L)\right\}$. Following \cite{nesetril-rodl-1989}, we write $\SS_\L$ for the class of $\L$-societies -- that class of all finite models of $\Sigma_\L$. (One easily verifies that $\SS_\L$ is an algebraically trivial \fraisse class.)
\end{defn}

\begin{obs}\label{Obs_HHrSSL}
Let $\L$ be a finite relational language in which all relation symbols have arity $\geq 2$, and let $r = \ari(\L)$; then $T_{\HH_r} \cleq T_{\SS_\L}$. To see this, just fix a relation $R\in\sig(\L)$ of arity $r$ -- then the reduct of the generic model $\A$ of $\SS_\L$ to the signature $\{R\}$ actually is the generic model of the class $\HH_r$ of $r$-hypergraphs.
\end{obs}

In the following theorem, we show that the converse is also true; i.e., the dividing-line corresponding to a class of societies is the same as the class of $r$-hypergraphs, where $r$ is the largest arity in the language.  This conclusion is formally stated in Corollary \ref{Cor_SocietiesDontMatter}.

\begin{thm}\label{Thm_SSLHHr}
Let $\L$ be a finite relational language in which all relation symbols have arity $\geq 2$, and let $r = \ari(\L)$. Then $T_{\SS_\L} \cleq T_{\HH_r}$.
\end{thm}
\begin{proof}
We assume that the signature of $\HH_r$ is $\{R^{(r)}\}$. Let $m = \sum_{Q\in\sig(\L)}\ari(Q)$, and let $(I_Q)_{Q\in\sig(\L)}$ be a partition of $m$ such that for each $Q^{(n)}\in\sig(\L)$, we have $I_Q = \left\{i_0(Q)<\cdots<i_{n-1}(Q)\right\}$ and an enumeration  $j_0(Q)<\cdots<j_{m-n-1}(Q)$ of $m\setminus I_Q$. For each $Q\in\sig(\L)$ of arity $n$, let 
$$\theta_Q(\xx_0,...,\xx_{n-1})=R(x_{0,i_0(Q)},...,x_{n-1,i_{n-1}(Q)},x_{0,j_0(Q)},...,x_{0,j_{r-n-1}(Q)}).$$
Given $C\in\SS_\L$, we define $B_C\in\HH_r$ as follows:
\begin{itemize}
\item $B_C = C\times m$
\item For $Q^{(n)}\in\sig(\L)$, we define the intermediate relation $R_Q^C$ to be the symmetric closure of 
$$\left\{\left((c_0,i_0(Q)),...,(c_{n-1},i_{n-1}(Q)),(c_0,j_0(Q)),...,(c_0,j_{r-n-1}(Q))\right):(c_0,...,c_{n-1})\in Q^C\right\}$$
Then we define $R^{B_C} = \bigcup_{Q\in\sig(\L)}R_Q^C$.
\end{itemize}
We define $u:C\to B_C^m$ by $u(c) = \big((c,0),...,(c,m-1)\big)$. Let $Q^{(n)}\in\sig(\L)$ and $c_0,...,c_{n-1}\in C$ be given. Firstly, we have 
\begin{align*}
C\models Q(c_0,...,c_{n-1}) &\implies B_C\models R{\left((c_0,i_0(Q)),...,(c_{n-1},i_{n-1}(Q)),(c_0,j_0(Q)),...,(c_0,j_{r-n-1}(Q))\right)}\\
&\iff B_C\models \theta_Q(u(c_0),...,u(c_{n-1}))
\end{align*}
Now, we claim that if 
$$B_C\models R{\left((c_0,i_0(Q)),...,(c_{n-1},i_{n-1}(Q)),(c_0,j_0(Q)),...,(c_0,j_{r-n-1}(Q))\right)},$$ then  $(c_0,...,c_{n-1})$ is in  $Q^C$. If not, then for some $Q_1^{(n_1)}\in\sig(\L)$ different from $Q$, we have 
$$\left((c_0,i_0(Q)),...,(c_{n_1-1},i_{n_1-1}(Q)),(c_0,j_0(Q)),...,(c_0,j_{r-n_1-1}(Q))\right)\in R^C_{Q_1}.$$
Since $Q_1^C$ contains only non-repeating tuples, it must then be that $I_{Q_1}\subseteq I_Q$ -- contradicting the fact that $(I_Q)_{Q\in\sig(\L)}$ is partition of $m$. Thus, $(c_0,...,c_{n-1})$ is in  $Q^C$, and we have proven that 
$$C\models Q(c_0,...,c_{n-1})\iff B_C\models \theta_Q(u(c_0),...,u(c_{n-1})).$$
As $C\in\SS_\L$ was arbitrary, we have shown that $T_{\SS_\L} \cleq T_{\HH_r}$.
\end{proof}

\begin{cor}\label{Cor_SocietiesDontMatter}
 Let $\L$ be a finite relational language in which all relation symbols have arity $\geq 2$.  Then, $\C_{\SS_\L} = \C_{\HH_{\ari(\L)}}$.
\end{cor}

\begin{proof}
 By Observation \ref{Obs_HHrSSL} and Theorem \ref{Thm_SSLHHr}.
\end{proof}

As we shall see in the following subsection, $\C_{\HH_{r+1}}$ is exactly equal to the theories that have $r$-IP.  One might hope that, by studying societies in general, one might find a whole zoo of generalizations of the independence property, akin to $r$-IP.  However, Corollary \ref{Cor_SocietiesDontMatter} says that no such thing seems to exist.  Adding more relation symbols (i.e., considering more formulas) does nothing to alter the positive local complexity of the theory.  In some sense, the $r$-IP's are the only generalizations of the independence property of this type.

\subsection{Multi-partite multi-concept classes}

In this subsection, we show that the irreducible dividing-line corresponding to $(r+1)$-hypergraphs is precisely the same as $r$-IP, the $r$-independence property.  We begin by discussing $p$-partite concept classes, which is a generalization of (standard) concept classes used to study the independence property.

\begin{defn}
If $W$ is a set, then $S\subseteq_mW$ means that $S$ is a multi-set all of whose elements are members of $W$, each with finite multiplicity.
\end{defn}

\begin{defn}
For $0<p<\omega$, $\L_p$ is the $(p+1)$-sorted language (with sorts $\sortS_0,\sortS_1,...,\sortS_p$) and just one relation symbol $R\subseteq \sortS_0\times\sortS_1\times\cdots\times\sortS_p$. The \fraisse class $\Fin(\L_p)$ of all finite $\L_p$-structures is also known as the class of $(p+1)$-partite $(p+1)$-hypergraphs.
\end{defn}

\begin{defn}
For $0<p<\omega$, a finite $p$-partite concept class is a multi-set $\cclass\subseteq_m\P(X_1\times\cdots\times X_{p})$. 

Given a finite $p$-partite concept class $\cclass\subseteq_m\P(X_1\times\cdots\times X_{p})$ (multi-set of subsets of $X_1\times\cdots\times X_p$), we define an $\L_p$-structure $B_\cclass$ with $\sortS_0(B_\cclass) = \cclass$, $\sortS_i(B_\cclass) = X_i$ for each $i=1,...,p$, and 
$$R^{B_\cclass} = \left\{(S,x_1,...,x_p):(x_1,...,x_p)\in S\in\cclass\right\}.$$
Let $\JJ_p$ be the isomorphism-closure of 
$$\left\{B_\cclass:\textnormal{$\cclass$ is a finite $p$-partite concept class}\right\}.$$
\end{defn}

\begin{obs}
For every $0<p<\omega$, $\JJ_p = \Fin(\L_p)$ and $\HH_{p+1}^* = \widetilde{\Fin(\L_p)}$ (in the notation of Subsection \ref{Subsec_one_sort}).
\end{obs}

We define the $r$-independence property ($r$-IP), which generalizes the usual independence property.  Then, in Proposition \ref{Prop_Chernikov}, we show that this exactly equals the dividing-line corresponding to $(r+1)$-hypergraphs.

\begin{defn}
Let $T\in\TT$, and let $\M\models T$ be $\aleph_1$-saturated. We say that $T$ has the {\em $r$-independence property} ($r$-IP) if there is a formula $\phi(x_0,...,x_{r-1};y)\in \L_T$ such that  there are $a_{0,i},...,a_{r-1,i}$ in $\M$ ($i<\omega$) such that for any finite $X\subset\omega^r$, there is some $b_X$ in $\M$ such that for all $(i_0,...,i_{n-1})\in\omega^r$,
$$\M\models \phi(a_{0,i_0},...,a_{r-1,i_{r-1}};b_X)\iff (i_0,...,i_{r-1})\in X.$$
Let 
$\mathbf{IP}_r=\left\{T\in\TT:\textnormal{$T$ has $r$-IP} \right\}.$
Note that the usual independence property is $1$-IP.
\end{defn}

\begin{prop}[\cite{chernikov-et-al}]\label{Prop_Chernikov}
$\mathbf{IP}_r=\C_{\HH_{r+1}}$.
\end{prop}
\begin{proof}
Let $T\in\TT$. If $T\in\C_{\HH_{r+1}}$ is witnessed by $F:A\to M$ (where $\M\models T$) and $\phi(x_0,...,x_{r-1},x_r)\in \L_T$, then it is easy to see that $\phi$ witnesses the fact that $T$ has the $r$-independence property. Conversely, if $T$ has $r$-IP, then there are  a formula $\phi(x_0,...,x_{r-1};y)\in \L_T$, $(a_{0,i},...,a_{r-1,i})_{i<\omega}$ and $(b_X)_{X\subset_\fin \omega^r}$ in some $\M\models T$ witnessing this. Then, if $\B$ is the generic model of $\JJ_r$, we have injections $u_i:\sortS_i(B)\to M$ ($i\leq r$) such that for all $b_0\in\sortS_0(\B),...,b_r\in\sortS_r(\B)$, $\B\models R(b_0,...,b_r)\iff \M\models\phi(u_0(b_0),...,u_r(b_r))$. By Theorem \ref{Thm_one_sort}, it follows that $T\in\C_{\tilde\JJ_r} = \C_{\HH^*_{r+1}} = \C_{\HH_{r+1}}$.
\end{proof}

Combining this with Corollary \ref{Cor_hyp_desc_chain}, we see that the $r$-independence property forms a strictly decreasing chain of irreducible dividing-lines.

\begin{cor}
Every $\mathbf{IP}_r$ ($2\leq r<\omega$) is an irreducible dividing-line, and 
$$\mathbf{IP}=\mathbf{IP}_1\supsetneq\mathbf{IP}_2\supsetneq\cdots\supsetneq\mathbf{IP}_r\supsetneq\cdots.$$
\end{cor}

The fact that these are strict was already well-known (e.g., \cite{chernikov-et-al}).  However, that each is an irreducible dividing-line is an interesting fact, providing evidence that our definition of irreducible is the ``right'' one.  Indeed, irreducibility really should encompass all known positive local dividing-lines.

\newpage
\section{Open questions}\label{Sec_Conjectures}


When looking at Section \ref{Sec_LO_Indisc}, one notices that, if $\KK$ is an algebraically trivial indecomposable \fraisse class and $T_\KK$ is unstable, then $\C_\KK$ is characterized by a collapse of indiscernibles when $\KK^<$ is a Ramsey class.  So a natural question arises:

\begin{ques}
 Let $\KK$ be an algebraically trivial indecomposable \fraisse class such that $T_\KK$ is unstable and $|S_1(T_\KK)| = 1$.  When is $\KK^<$ a Ramsey class?  Is there a model-theoretic characterization of this?
\end{ques}

Also in that section, one notices a difference between our general result for collapse-of-indiscernibles (Theorem \ref{Thm_collapse_char_2}) and the specific results found in the literature (e.g., \cite{chernikov-et-al, guingona-hill-scow, scow-2011}).

\begin{ques}
 Suppose $\KK$ is an algebraically trivial indecomposable \fraisse class such that $\KK^<$ is a Ramsey class.  Can one find a specific reduct $T_0$ of $T_\KK$ such that a theory $T$ lies outside of $\C_\KK$ if and only if every $\KK^<$-indiscernible in $T$ collapses to $T_0$?
\end{ques}



We would like to better understand the quasi-ordering $\cleq$ and the irreducible dividing-lines it generates.  For example, deciding which classes $\KK$ are equivalent vis-\`{a}-vis the class $\C_\KK$ seems to be an interesting project.  Which are equivalent to the trivial dividing-line?

\begin{ques}
 Suppose $\KK$ is an indecomposable algebraically trivial \fraisse class such that $|S_1(T_\KK)|=1$ and $T_\KK$ is stable.  Then, do we have $\C_\KK = \TT$?  If not, can we characterize which $\KK$ yield the trivial dividing line?
 
 (Obviously, if $\KK_=$ is the \fraisse class of finite pure sets (in the empty signature); then $\C_{\KK_=}=\TT$.)
\end{ques}


Another question revolves around the number of irreducible dividing lines.  By Corollary \ref{Cor_InfiniteClasses}, we know there are between $\aleph_0$ and $2^{\aleph_0}$ such, but can we get a better estimate?

\begin{ques}
 Is the set of irreducible dividing-lines countable?
\end{ques}



During the first attempt at categorizing irreducibility for classes of theories, we replaced ``completeness'' with ``countable completeness'' in Definition \ref{Def_irred_div-line}.  Although the proof of Lemma \ref{Lemma_have_minimal} seems to require at least ``$(2^{\aleph_0})^+$-completeness,'' is this actually necessary?

\begin{ques}
 Let $\C\subset\TT$ be a prime filter class.  Is $\C$ complete if and only if $\C$ is countably-complete (i.e., every descending $\cleq$-chain $(T_n)_{n<\omega}$ of members of $\C$, there is some $T\in\C$ such that $T\cleq T_n$ for all $n<\omega$)?
\end{ques}

Notice that $\cleq$ relates any sort of one theory to any sort of another (which is why, in Proposition \ref{Prop_MOandLO}, we find that $\C_{\MO_k} = \C_{\mathbf{LO}}$ for all $k > 0$).  What would happen if one restricted the sorts to compare?  For example, could one recover a generalized collapse-of-indiscernible result on sorts (or partial types) akin to the one for op-dimension in \cite{guingona-hill-scow}?  This may be related to examining the witness number from Observation \ref{Obs_witnessing}.  We hope to explore this (and the other questions in this section) in future papers.

\newpage

\bibliographystyle{plain}
\bibliography{myref}

\bigskip

\noindent{\em Author's addresses:} 
\begin{itemize}
\item[] C. Donnay Hill, Department of Mathematics and Computer Science,
Wesleyan University,
655 Exley Science Tower,
265 Church Street,
Middletown, CT 06459

\item[] V. Guingona, Department of Mathematics,
Towson University,
8000 York Road,
Towson, MD 21252
\end{itemize}

\end{document}